\documentclass[twoside,11pt]{article}

\usepackage{jmlr2e}

\usepackage{enumerate}
\usepackage[usenames,dvipsnames]{xcolor}
\usepackage{tikz}
\usetikzlibrary{chains}
\usetikzlibrary{fit}
\usepackage{pgflibraryarrows}		
\usepackage{pgflibrarysnakes}		
\usepackage{nicefrac} 

\usepackage{epsfig}
\usetikzlibrary{shapes.symbols,patterns} 
\usepackage{pgfplots}
\usepackage{makecell}
\usepackage{mathtools}
\DeclarePairedDelimiter{\ceil}{\lceil}{\rceil}

\usepackage{dsfont}

\usepackage{mathtools}
\usepackage{hyperref}
\hypersetup{colorlinks=true,citecolor=blue,linkcolor=blue,filecolor=blue,urlcolor=blue,breaklinks=true}
\usepackage{todonotes}

\usepackage{footnote}

\usepackage{subfig}

\usepackage{tablefootnote}

\newcommand{\prlsection}[1]{{\bf{#1}.}}

\newcommand{\norm}[1]{\left\lVert#1\right\rVert}

\newcommand{\E}[1]{\,{\mathds E}\!\left[#1\right]} 

\DeclareMathOperator{\st}{s.t.}

\newcommand{\drv}{\ensuremath{\mathrm{d}}}

\newcommand{\inprod}[2]{\ensuremath{\left\langle{#1}\vphantom{\big|},\vphantom{\big|}{#2}\right\rangle}}
\newcommand{\transp}{\ensuremath{^{\scriptscriptstyle{\top}}}}

\newcommand{\A}{\mathbb{A}}
\newcommand{\AAA}{\mathcal{A}}

\newcommand{\K}{\mathbb{K}}
\newcommand{\N}{\mathbb{N}}

\newcommand{\M}{\mathbb{M}}

\newcommand{\proj}{\Pi}

\newcommand{\opt}{^\star}

\newcommand{\MM}{\mathbb{M}}

\newcommand{\Borelsigalg}[1]{\ensuremath{\mathcal{B}\!\left(#1\right)}}

\newcommand{\KL}[2]{\ensuremath{\mathsf{D}\hspace{0.5mm}{\vphantom{T}}\!\!\left ({#1}\vphantom{\big|}\vert \hspace{-0.4mm} \vert \vphantom{\big|}{#2}\right)}}

\newcommand{\R}{\ensuremath{\mathbb{R}}}

\newcommand{\Rsp}{\ensuremath{\R_{> 0}}}

\newcommand{\XX}{\mathbb{X}}
\newcommand{\X}{\mathbb{X}}

\newcommand{\Lp}[1]{\mathrm{L}^{#1}}

\DeclareMathOperator{\T}{\mathbb{T}}
\DeclareMathOperator{\yeta}{y^\star_\zeta}
\newtheorem{myass}{Assumption} 

\usepackage{lastpage}

\jmlrheading{20}{2019}{1-\pageref{LastPage}}{8/17; Revised
1/19}{9/19}{17-486}{}

\ShortHeadings{Generalized maximum entropy estimation}{Sutter, Sutter, Mohajerin Esfahani and Lygeros}
\firstpageno{1}

\begin{document}

\title{Generalized maximum entropy estimation}

\author{\name Tobias Sutter \email tobias.sutter@epfl.ch \\
       \addr Risk Analytics and Optimization Chair \\
       EPFL, Switzerland
       \AND
       \name David Sutter \email suttedav@phys.ethz.ch \\
       \addr Institute for Theoretical Physics\\
       ETH Zurich, Switzerland
        \AND
       \name Peyman Mohajerin Esfahani \email P.MohajerinEsfahani@tudelft.nl \\
       \addr Delft Center for Systems and Control\\
       TU Delft, The Netherlands
       \AND
       \name John Lygeros \email lygeros@control.ee.ethz.ch \\
       \addr Department of Electrical Engineering and Information Technology\\
       ETH Zurich, Switzerland}

\editor{Benjamin Recht}

\maketitle

\begin{abstract}
We consider the problem of estimating a probability distribution that maximizes the entropy while satisfying a finite number of moment constraints, possibly corrupted by noise. Based on duality of convex programming, we present a novel approximation scheme using a smoothed fast gradient method that is equipped with explicit bounds on the approximation error. We further demonstrate how the presented scheme can be used for approximating the chemical master equation through the zero-information moment closure method, and for an approximate dynamic programming approach in the context of constrained Markov decision processes with uncountable state and action spaces.
\end{abstract}

\vspace{2mm}


\section{Introduction}

This article investigates the problem of estimating an unknown probability distribution given a finite number of observed moments that might be corrupted by noise. Given that the observed moments are consistent, i.e., there exists a probability distribution that satisfies all the moment constraints, the problem is underdetermined and has infinitely many solutions. This raises the question of which solution to choose. A natural choice would be to pick the one with the highest entropy, called the \emph{MaxEnt distribution}. The main reason why the MaxEnt distribution is a natural choice is due to a concentration phenomenon described by~\cite{ref:Jaynes-03}:

\vspace{2mm}
\begin{minipage}{36.0em}
``\emph{If the information incorporated into the maximum-entropy analysis includes all the constraints actually operating in the random experiment, then the distribution predicted by maximum entropy is overwhelmingly the most likely to be observed experimentally.}'' 
\end{minipage}
\vspace{2mm}

\noindent See~\cite{ref:Jaynes-03, ref:Grunwald-08} for a rigorous statement. This maximum entropy estimation problem subject to moment constraints, also known as the \emph{principle of maximum entropy}, is applicable to large classes of problems in natural and social sciences --- in particular in economics, see~\cite{ref:Golan-08} for a comprehensive survey.
Furthermore it has important applications in approximation methods to dynamical objects, such as in systems biology, where MaxEnt distributions are key objects in the so-called \emph{moment closure method} to approximate the chemical master equation~\cite{ref:Smadbeck-13}, or more recently in the context of approximating dynamic programming \cite{ref:MohSut-17} where MaxEnt distributions act as a regularizer, leading to computationally more efficient optimization programs.

 Their operational significance motivates the study of numerical methods to compute MaxEnt distributions, which are the solutions of an infinite-dimensional convex optimization problem and as such computationally intractable in general. Since it was shown that the MaxEnt distribution subject to a finite number of moment constraints (if it exists) belongs to the exponential family of distributions~\cite{ref:Csiszar-75}, its computation can be reduced to solving a system of nonlinear equations, whose dimension is equal to the number of moment constraints \cite{ref:Mead-84}. Furthermore, the system of nonlinear equations involves evaluating integrals over the support set $\K$ that are computationally difficult in general. Even if  $\K$ is finite, finding the MaxEnt distribution is not straightforward, since solving a system of nonlinear equations can be computationally demanding. 

In this article, we present a new approximation scheme to minimize the relative entropy subject to noisy moment constraints. This is a generalization of the introduced maximum entropy problem and extends the principle of maximum entropy to the so-called \emph{principle of minimum discriminating information}~\cite{ref:Kullback-59}. We show that its dual problem exhibits a particular favorable structure that allows us to apply Nesterov's smoothing method~\cite{nesterov05} and hence tackle the presented problem using a fast gradient method
obtaining process convergence properties, unlike \cite{ref:Lasserre-11}.

Computing the MaxEnt distribution has applications in randomized rounding and the design of approximation algorithms. More precisely, it has been shown how to improve the approximation ratio for the symmetric and the asymmetric traveling salesman problem via MaxEnt distributions~\cite{asadpour10,oveis11}. Often, it is important to efficiently compute the MaxEnt distribution. For example, the zero-information moment closure method~\cite{ref:Smadbeck-13} (see Section~\ref{sec:moment:closure}), a recent approximate dynamic programming method for constrained Markov decision processes (see Section~\ref{sec:ADP}), as well as the approximation of the channel capacity of a large class of memoryless channels~\cite{TobiasSutter15} deal with iterative algorithms that require the numerical computation of the MaxEnt distribution in each iteration step.

\vspace{2mm}
\prlsection{Related results} Before comparing the approach presented in this article with existing methods we provide a brief digression on the moment problem. 
Consider a one-dimensional moment problem formulated as follows: Given a set $\K\subset \R$ and a sequence $(y_i)_{i\in\N}\subset\R$ of moments, does there exist a measure $\mu$ supported on $\K$ such that
\begin{equation} \label{eq:moment:intro}
y_i = \int_\K x^i \mu(\drv x) \quad \text{for all } i\in\N \ ?
\end{equation}
For $\K=\R$ and $\K=[a,b]$ with $-\infty<a<b<\infty$ the above moment problem is known as the \emph{Hamburger moment problem} and \emph{Hausdorff moment problem}, respectively. If the moment sequence is finite, the problem is called a \emph{truncated moment problem}. In both full and truncated cases, a measure $\mu$ that satisfies \eqref{eq:moment:intro}, is called a \emph{representing measure} of the sequence $(y_i)_{i\in\N}$. If a representing measure is unique, it is said to be \emph{determined by its moments}.
From the Stone-Weierstrass theorem it followes directly that every non-truncated representing measure with compact support is determined by its moments.
In the Hamburger moment problem, given a representing measure $\mu$ for a moment sequence $(y_i)_{i\in\N}$, a sufficient condition for $\mu$ being determined by its moments is the so-called \emph{Carleman condition}, i.e.,  $\sum_{i=1}^{\infty}y_{2i}^{-\nicefrac{1}{2i}}=\infty$. Roughly speaking this says that the moments should not grow too fast, see~\cite{ref:Akhiezer-65} for further details. For the Hamburger and the Hausdorff moment problem, there are necessary and sufficient conditions for the existence of a representing measure for a given moment sequence $(y_i)_{i\in\N}$ in both the full as well as the truncated setting, that exploit the rich algebraic connection with Hankel matrices see~\cite[Theorems~3.2, 3.3, 3.4]{ref:Lasserre-11}.

In \cite[Section~12.3]{ref:Lasserre-11} it is shown that the maximum entropy subject to finite moment constraints can be approximated by using duality of convex programming. The problem can be reduced to an unconstrained finite-dimensional convex optimization problem and an approximation hierarchy of its gradient and Hessian in terms of two single semidefinite programs involving two linear matrix inequalities is presented. The desired accuracy is controlled by the size of the linear matrix inequalities. The method seems to be powerful in practice, however a rate of convergence
has not been proven. Furthermore, it is not clear how the method extends to the case of uncertain moment constraints. 
In a finite dimensional setting, \cite{ref:Dudik-07} presents a treatment of the maximum entropy principle with generalized regularization measures, that as a special case contains the setting presented here. However, convergence rates of algorithms presented are not known and again it is not clear how the method extends to the case of uncertain moment constraints. The discrete case, where the support set $\K$ is discrete, has been studied in more detail in the past. It has been shown the the maximum entropy problem in the discrete case has a succinct description that is polynomial-size in the input and can be efficiently computed~\cite{singh14,straszak17}. Furthermore it was shown that the maximum entropy problem is equivalent to the counting problem~\cite{singh14}.

\vspace{2mm}
\prlsection{Structure} The layout of this paper is as follows: In Section~\ref{sec:problem:statement} we formally introduce the problem setting. Our results on an approximation scheme in a continuous setting are reported in Section~\ref{sec:rate:distortion}. In Section~\ref{sec:finite:dim:case}, we show how these results simplify in the finite-dimensional case. Section~\ref{sec:gradient:evalutaion} discusses the gradient approximation that is the dominant step of the proposed approximation method from a computational perspective. The theoretical results are applied in Section~\ref{sec:moment:closure} to the zero-information moment closure method and in Section~\ref{sec:ADP} to constrained Markov decision processes. We conclude in Section~\ref{sec:conclusion} with a summary of our work and comment on possible subjects of further research.

\vspace{2mm}
\prlsection{Notation}
The logarithm with basis 2 and $\mathrm{e}$ is denoted by $\log(\cdot)$ and $\ln(\cdot)$, respectively. We define the standard $n-$simplex as $\Delta_{n}:=\{  x\in\R^{n} : x\geq 0, \sum_{i=1}^{n} x_{i}=1\}$. For a probability mass function $p \in \Delta_{n}$ we denote its entropy by $H(p):=\sum_{i=1}^n -p_i \log p_i$. 
Let $B(y,r):=\{x\in\R^n \ : \ \|x-y\|_2 \leq r \}$ denote the ball with radius $r$ centered at $y$. Throughout this article, measurability always refers to Borel measurability.
For a probability density $p$ supported on a measurable set $B\subset \R$ we denote the differential entropy by $h(p):=-\int_{B} p(x) \log p(x) \drv x$. For $\A\subset\R$ and $1\leq p \leq \infty$, let $\Lp{p}(\A)$ denote the space of $\Lp{p}$-functions on the measure space $(\A, \Borelsigalg{\A}\!, \drv x)$, where $\Borelsigalg{\A}$ denotes the Borel $\sigma$-algebra and $\drv x$ the Lebesgue measure.  
Let $\XX$ be a compact metric space, equipped with its  Borel $\sigma$-field $\mathcal{B}(\cdot)$. The space of all probability measures on $(\XX, \mathcal{B}(\XX))$ will be denoted by $\mathcal{P}(\XX)$. The \emph{relative entropy} (or
Kullback-Leibler divergence) between any two probability measures $\mu, \nu \in \mathcal{P}(\XX)$ is defined by
\begin{equation*}
\KL{\mu}{\nu} := \left\{ \begin{array}{ll}
\int_\XX \log\left( \frac{\drv\mu}{\drv \nu} \right) \drv\mu, &\text{if } \mu \ll \nu\\
+\infty, & \text{otherwise} \, ,
\end{array} \right.
\end{equation*}
where $\ll$ denotes absolute continuity of measures, and $\tfrac{\drv\mu}{\drv \nu}$ is the Radon-Nikodym derivative. The relative entropy is non-negative, and is equal to zero if and only if $\mu\equiv \nu$.
Let $\XX$ be restricted to a compact metric space and let us consider the pair of vector spaces
$(\M(\XX),\mathbb{B}(\XX))$ where $\M(\XX)$ denotes the space of finite signed measures on $\mathcal{B}(\XX)$ and $\mathbb{B}(\XX)$ is the Banach space of bounded measurable functions on $\XX$ with respect to the sup-norm and consider the bilinear form
\begin{align*}
\inprod{\mu}{f}:=\int_\XX f(x)\mu(\drv x).
\end{align*}
 This induces the total variation norm as the dual norm on $\M(\XX)$, since by~\cite[p.2]{ref:Hernandez-99}
\begin{equation*}
\| \mu \|_* = \sup_{\|f\|_\infty \leq 1}\inprod{\mu}{f} = \|\mu\|_{\mathsf{TV}},
\end{equation*}
making $\M(\XX)$ a Banach space.
In the light of~\cite[p.~206]{ref:Hernandez-99} this is a dual pair of Banach spaces; we refer to~\cite[Section~3]{anderson87} for the details of the definition of dual pairs. The Lipschitz norm is defined as $\|u\|_L := \sup_{x,x'\in \X} \{ |u(x)|, \frac{|u(x)-u(x')|}{\|x-x'\|_\infty}\}$ and $\mathcal{L}(\X)$ denotes the space of Lipschitz functions on $\X$.

\section{Problem Statement}  \label{sec:problem:statement}

Let $\K \subset \R$ be compact and consider the scenario where a probability measure $\mu \in \mathcal{P}(\K)$ is unknown and only observed via the following measurement model
\begin{equation} \label{eq:measurement:model}
y_{i} = \inprod{\mu}{x^{i}} + u_{i}, \quad u_{i} \in \mathcal{U}_{i} \quad \text{for }i=1,\hdots,M \, ,
\end{equation}
where $u_{i}$ represents the uncertainty of the obtained data point $y_{i}$ and $\mathcal{U}_{i}\subset\R$ is compact, convex and $0\in\mathcal{U}_{i}$ for all $i=1,\hdots,M$.
Given the data $(y_{i})_{i=1}^M\subset \R$, the goal is to estimate a probability measure $\mu$ that is consistent with the measurement model \eqref{eq:measurement:model}. This problem (given that $M$ is finite) is underdetermined and has infinitely many solutions. Among all possible solutions for \eqref{eq:measurement:model}, we aim to find the solution that maximizes the entropy.
Define the set $T:=\times_{i=1}^{M} \{y_{i}-u  :  u\in \mathcal{U}_{i} \}\subset \R^{M}$ and the linear operator $\mathcal{A}:\MM(\K)\to \R^{M}$ by 
\begin{align*}
(\mathcal{A}\mu)_{i}:=\inprod{\mu}{x^{i}} = \int_{\K}x^{i}\mu(\drv x) \quad \text{for all} \quad i=1,\hdots, M\, .
\end{align*}
The operator norm is defined as $\norm{\mathcal{A}}:=\sup_{\norm{\mu}_{\mathsf{TV}}=1, \norm{y}_2=1} \inprod{\mathcal{A}\mu}{y}$. Note that due to the compactness of $\K$ the operator norm is bounded, see Lemma~\ref{lem:lipschitz:cts:gradient} for a formal statement.
The adjoint operator to $\mathcal{A}$ is given by $\mathcal{A}^{*}:\R^{M}\to\mathbb{B}(\K)$, where $\mathcal{A}^{*}z (x):=\sum_{i=1}^{M}z_{i}x^{i}$; note that the domain and image spaces of the adjoint operator are well defined as $(\mathbb{B}(\K),\MM(\K))$ is a topological dual pairs and the operator $\mathcal{A}$ is bounded \cite[Proposition 12.2.5]{ref:Hernandez-99}. 

Given a reference measure $\nu \in \mathcal{P}(\K)$, the problem of minimizing the relative entropy subject to moment constraints \eqref{eq:measurement:model} can be formally described by
 \begin{align} \label{eq:main:problem}
 	       \quad J^{\star}= \min\limits_{\mu\in\mathcal{P}(\K)} \left\{ \KL{\mu}{\nu} \ : \ \mathcal{A}\mu \in T \right\}.
 	\end{align}
We note that the reference measure $\nu \in \mathcal{P}(\K)$ is always fixed a priori. A typical choice for $\nu$ is the uniform measure over $\K$.

	\begin{proposition}[Existence \& uniqueness of \eqref{eq:main:problem}] \label{prop:solvability}
	The optimization problem~\eqref{eq:main:problem} attains an optimal feasible solution that is unique.
	\end{proposition}
\begin{proof}
The variational representation of the relative entropy~\cite[Corollary~4.15]{ref:boucheron-13} implies that the mapping $\mu \mapsto \KL{\mu}{\nu}$  is lower-semicontinuous~\cite{ref:Luenberger-69}.
Note also that the space of probability measures on $\K$ is compact~\cite[Theorem~15.11]{ref:aliprantis-07}. Moreover, since the linear operator $\AAA$ is bounded, it is continuous. As a result, the feasible set of problem \eqref{eq:main:problem} is compact and hence the optimization problem attains an optimal solution. Finally, the strict convexity of the relative entropy~\cite{ref:Csiszar-75} ensures uniqueness of the optimizer.
\end{proof}

Note that if $\mathcal{U}_i=\{0\}$ for all $i=1,\ldots,M$, i.e., there is no uncertainty in the measurement model~\eqref{eq:measurement:model}, Proposition~\ref{prop:solvability} reduces to a known result~\cite{ref:Csiszar-75}. Consider the special case where the reference measure $\nu$ is the uniform measure on $\K$ and let $p$ denote the Radon-Nikodym derivative $\tfrac{\drv \mu}{\drv \nu}$ (whose existence can be assumed without loss of generality). Since $\mathcal{A}$ is weakly continuous and the differential entropy is known to be weakly lower semi-continuous \cite{ref:boucheron-13}, we can restrict attention to a (weakly) dense subset of the feasible set and hence assume without loss of generality that $p\in\Lp{1}(\K)$.
Problem \eqref{eq:main:problem} then reduces to
 \begin{align} \label{eq:main:problem:special case}
				\max\limits_{p\in\Lp{1}(\K)}	\left \lbrace h(p) \, : \,  \int_{\K} p(x) \drv x = 1, \ \int_{\K} x^i p(x) \drv x \in T_i, \ \forall i=1,\hdots,M \right \rbrace . 
 	\end{align} 
Problem \eqref{eq:main:problem:special case} is a generalized maximum entropy estimation problem that, in case $\mathcal{U}_i=\{0\}$ for all $i=1,\hdots,M$, simplifies to the standard entropy maximization problem subject to $M$ moment constraints.
In this article, we present a new approach to solve \eqref{eq:main:problem} that is based on its dual formulation. It turns out that the dual problem of \eqref{eq:main:problem} has a particular structure that allows us to apply Nesterov's smoothing method~\cite{nesterov05} to accelerate convergence. Furthermore, we will show how an $\varepsilon$-optimal solution to \eqref{eq:main:problem} can be reconstructed. This is done by solving the dual problem of \eqref{eq:main:problem}. To achieve additional feasibility guarantees for the $\varepsilon$-optimal solution to \eqref{eq:main:problem}, we introduce and a second smoothing step that is motivated by~\cite{ref:devolder-12}. The problem of entropy maximization subject to uncertain moment constraints \eqref{eq:main:problem:special case} can be seen as a special case of \eqref{eq:main:problem}.

\section{Relative entropy minimization}  \label{sec:rate:distortion}
We start by recalling that an unconstrained minimization of the relative entropy with an additional linear term in the cost admits a closed form solution. Let $c\in\mathbb{B}(\K)$, $\nu\in\mathcal{P}(\K)$ and consider the optimization problem
\begin{equation} \label{eq:modified:entropy:max}
\min\limits_{\mu\in\mathcal{P}(\K)}  \left \{ \KL{\mu}{\nu} - \inprod{\mu}{c} \right\}.
\end{equation}
\begin{lemma}[Gibbs distribution] \label{lem:entropy:max}
The unique optimizer to problem \eqref{eq:modified:entropy:max} is given by the Gibbs distribution, i.e.,
\begin{equation*}
\mu^{\star}(\drv x) = \frac{2^{c(x)}\nu(\drv x)}{\int_{\K}2^{c(x)}\nu(\drv x)} \quad\text{for } x \in \K  ,
\end{equation*}
which leads to the optimal value of $- \log  \int_{\K}2^{c(x)}\nu(\drv x) $.
\end{lemma}
\begin{proof}
The result is standard and follows from~\cite{ref:Csiszar-75} or alternatively by~\cite[Lemma~3.10]{TobiasSutter15}.
\end{proof}

Let $\R^{M}\ni z\mapsto \sigma_{T}(z):=\max_{x\in T}\inprod{x}{z}\in\R$ denote the support function of $T$, which is continuous since $T$ is compact \cite[Corollary~13.2.2]{ref:Rockafellar-97}. The primal-dual pair of problem \eqref{eq:main:problem} can be stated as
\begin{align}
\text{(primal program)}: \quad J^{\star} &= \min\limits_{\mu\in\mathcal{P}(\K)} \Big \{  \KL{\mu}{\nu} + \sup_{z\in\R^{M}}\left\{ \inprod{\mathcal{A}\mu}{z} - \sigma_{T}(z)\right\} \Big \} \label{eq:primal:problem}\\
\text{(dual program)}: \quad J_{\mathsf{D}}^{\star} &= \sup_{z\in\R^{M}} \Big \{ - \sigma_{T}(z)  +   \min\limits_{\mu\in\mathcal{P}(\K)} \left\{ \KL{\mu}{\nu}  + \inprod{\mathcal{A}\mu}{z}\right\} \Big \} \, ,  \label{eq:dual:problem}
\end{align}
where the dual function is given by 
\begin{equation} \label{eq:dual:function}
F(z)= - \sigma_{T}(z)  +   \min\limits_{\mu\in\mathcal{P}(\K)} \left\{ \KL{\mu}{\nu}  + \inprod{\mathcal{A}\mu}{z}\right\}.
\end{equation}
Note that the primal program \eqref{eq:primal:problem} is an infinite-dimensional convex optimization problem. The key idea of our analysis is driven by Lemma~\ref{lem:entropy:max} indicating that the dual function, that involves a minimization running over an infinite-dimensional space, is analytically available. As such, the dual problem becomes an unconstrained finite-dimensional convex optimization problem, which is amenable to first-order methods.
\begin{lemma}[Zero duality gap] \label{lem:zero:duality:gap}
There is no duality gap between the primal program \eqref{eq:primal:problem} and its dual \eqref{eq:dual:problem}, i.e., $J^{\star}=J_{\mathsf{D}}^{\star}$. Moreover, if there exists $\bar{\mu}\in\mathcal{P}(\K)$ such that $\mathcal{A}\bar{\mu}\in\mathsf{int}(T)$, then the set of optimal dual variables in \eqref{eq:dual:problem} is compact.
\end{lemma}
\begin{proof}
Recall that the relative entropy is known to be lower semicontinuous and convex in the first argument, which can be seen as a direct consequence of the duality relation for the relative entropy~\cite[Corollary~4.15]{ref:boucheron-13}. Hence, the desired zero duality gap follows by Sion's minimax theorem~\cite[Theorem~4.2]{ref:Sion-58}. The compactness of the set of dual optimizers is due to~\cite[Proposition~5.3.1]{ref:Bertsekas-09}. 
\end{proof}

Because the dual function \eqref{eq:dual:function} turns out to be non-smooth, in the absence of any additional structure, the efficiency estimate of a black-box first-order method is of order $O( \nicefrac{1}{\varepsilon^{2}})$, where $\varepsilon$ is the desired absolute additive accuracy of the approximate solution in function value~\cite{ref:nesterov-book-04}. 
We show, however, that the generalized entropy maximization problem \eqref{eq:primal:problem} has a certain structure that allows us to deploy the recent developments in \cite{nesterov05} for approximating non-smooth problems by smooth ones, leading to an efficiency estimate of  order $O( \nicefrac{1}{\varepsilon})$.  This, together with the low complexity of each iteration step in the approximation scheme, offers a numerical method that has an attractive computational complexity. 
In the spirit of~\cite{nesterov05,ref:devolder-12}, we introduce a smoothing parameter $\eta:=(\eta_{1},\eta_{2})\in\Rsp^{2}$ and consider a smooth approximation of the dual function 
\begin{equation} \label{eq:dual:function:smoothed}
F_{\eta}(z):=  -\max_{x\in T}\left\{ \inprod{x}{z} - \frac{\eta_{1}}{2}\norm{x}_{2}^{2} \right\} +  \min\limits_{\mu\in\mathcal{P}(\K)} \left\{ \KL{\mu}{\nu} + \inprod{\mathcal{A}\mu}{z}\right\}-\frac{\eta_{2}}{2}\norm{z}_{2}^{2} \, ,
\end{equation}
with respective optimizers denoted by $x^{\star}_{z}$ and $\mu^{\star}_{z}$. Consider the projection operator $\pi_T:\R^m\to\R$, $\pi_T(y)=\arg\min_{x\in T}\|x-y\|_2^2$. It is straightforward to see that the optimizer $x^{\star}_{z}$ is given by
\begin{align*}
x^{\star}_{z} = \arg\min_{x\in T} \| x - \eta_1^{-1}z\|_2^2 = \pi_T\left(\eta_1^{-1}z\right). 
\end{align*}
Hence, the complexity of computing $x^\star_z$ is determined by the projection operator onto $T$; for simple enough cases (e.g., 2-norm balls, hybercubes) the solution is analytically available, while for more general cases (e.g., simplex, 1-norm balls) it can be computed at relatively low computational effort, see \cite[Section~5.4]{richter_phd} for a comprehensive survey. 
The optimizer $\mu^{\star}_{z}$ according to Lemma~\ref{lem:entropy:max} is given by
\begin{equation*}
\mu^{\star}_{z}(B)  = \frac{\int_B 2^{-\mathcal{A}^*z(x)}\nu(\drv x)}{\int_\K 2^{-\mathcal{A}^*z(x)}\nu(\drv x)}, \quad \text{for all } B\in\mathcal{B}(\K).
\end{equation*}
\begin{lemma}[Lipschitz gradient] \label{lem:lipschitz:cts:gradient}
The dual function $F_{\eta}$ defined in \eqref{eq:dual:function:smoothed} is $\eta_{2}$-strongly concave and differentiable. Its gradient $\nabla F_{\eta}(z) =-x^{\star}_{z} + \mathcal{A}\mu^{\star}_{z} - \eta_{2} z$ is Lipschitz continuous with Lipschitz constant $\tfrac{1}{\eta_{1}}+\left( \sum_{i=1}^M B^i \right)^2+\eta_{2}$ and $B:=\max\{|x| \ : \ x\in \K\}$.
\end{lemma}
\begin{proof}
The proof follows along the lines of~\cite[Theorem~1]{nesterov05} and in particular by recalling that the relative entropy (in the first argument) is strongly convex with convexity parameter one and Pinsker's inequality, that says that for any $\mu\in\mathcal{P}(\K)$ we have
\begin{align}
 \|\mu-\nu\|_{\textsf{TV}}  \leq \sqrt{2\KL{\mu}{\nu}} \, .
\end{align}
Moreover, we use the bound
\begin{align}
\norm{\AAA} &=		\sup\limits_{\lambda\in\R^{M}\!, \, \mu\in\mathcal{P}(\K)} \left\{ \inprod{\AAA \mu}{\lambda} \ : \ \norm{\lambda}_{2}=1, \ \norm{\mu}_{\mathsf{TV}}=1 \right\} \nonumber \\
			  &\leq 	\sup\limits_{\lambda\in\R^{M}\!, \, \mu\in\mathcal{P}(\K)} \left\{ \norm{\AAA \mu}_{2} \norm{\lambda}_{2} \ : \ \norm{\lambda}_{2}=1, \ \norm{\mu}_{\mathsf{TV}}=1\right\}\label{eq:norm:W:proof:step:CS} \\
			&\leq 	   \sup\limits_{ \mu\in\mathcal{P}(\K)} \left\{ \norm{\AAA \mu}_{1} \ : \ \norm{\mu}_{\mathsf{TV}}=1\right\}\nonumber \\
			&= 		\sup\limits_{ \mu\in\mathcal{P}(\K)} \left\{ \sum_{i=1}^{M} \left| \int_{\mathcal{\K}} x^i \mu(\drv x) \right|   \ : \ \norm{\mu}_{\mathsf{TV}}=1\right\}\nonumber \\
			&\leq 	\sum_{i=1}^M B^i \, , \nonumber
\end{align}
where \eqref{eq:norm:W:proof:step:CS} is due to the Cauchy-Schwarz inequality.
\end{proof}

Note that $F_{\eta}$ is $\eta_{2}$-strongly concave and according to Lemma~\ref{lem:lipschitz:cts:gradient} its gradient is Lipschitz continuous with constant $L(\eta):=\tfrac{1}{\eta_{1}}+\norm{\mathcal{A}}^{2}+\eta_{2}$.
We finally consider the approximate dual program given by
\begin{align} 
\text{(smoothed dual program)}: \quad J_{\eta}^{\star} &= \sup_{z\in\R^{M}} F_{\eta}(z) \, .  \label{eq:dual:problem:approx:double}
\end{align}
It turns out that \eqref{eq:dual:problem:approx:double} belongs to a favorable class of smooth and strongly convex optimization problems that can be solved by a fast gradient method given in Algorithm~\hyperlink{algo:1}{1} (see~\cite{ref:nesterov-book-04}) with an efficiency estimate of the order $O(\nicefrac{1}{\sqrt{\varepsilon}})$.

 \begin{table}[!htb]
\centering 
\begin{tabular}{c}
  \Xhline{3\arrayrulewidth}  \hspace{1mm} \vspace{-3mm}\\ 
\hspace{0mm}{\bf{\hypertarget{algo:1}{Algorithm 1: } }} Optimal scheme for smooth $\&$ strongly convex optimization \hspace{14mm} \\ \vspace{-3mm} \\ \hline \vspace{-0.5mm}
\end{tabular} \\
\vspace{-5mm}
 \begin{flushleft}
  {\hspace{3mm}Choose $w_0=y_{0} \in \R^{M}$ and $\eta\in\Rsp^2$}
 \end{flushleft}
 \vspace{-6mm}
 \begin{flushleft}
  {\hspace{3mm}\bf{For $k\geq 0$ do$^{*}$}}
 \end{flushleft}
 \vspace{-8mm}
 
  \begin{tabular}{l l}
\hspace{15mm}{\bf Step 1: } & Set $y_{k+1}=w_{k}+\frac{1}{L(\eta)}\nabla F_{\eta}(w_{k})$ \\
\hspace{15mm}{\bf Step 2: } & Compute $w_{k+1}=y_{k+1} + \frac{\sqrt{L(\eta)}-\sqrt{\eta_{2}}}{\sqrt{L(\eta)}+\sqrt{\eta_{2}}}(y_{k+1}-y_{k})$\\
  \end{tabular}
   \begin{flushleft}
  {\hspace{3mm}[*The stopping criterion is explained in Remark~\ref{remark:stopping}]}
  \vspace{-10mm}
 \end{flushleft}  
\begin{tabular}{c}
\hspace{1mm} \phantom{ {\bf{Algorithm:}} Optimal Scheme for Smooth $\&$ Strongly Convex Optimization}\hspace{15mm} \\ \vspace{-1.0mm} \\\Xhline{3\arrayrulewidth}
\end{tabular}
\end{table}

Under an additional regularity assumption, solving the smoothed dual problem~\eqref{eq:dual:problem:approx:double} provides an estimate of the primal and dual variables of the original non-smooth problems \eqref{eq:primal:problem} and \eqref{eq:dual:problem}, respectively, as summarized in the next theorem (Theorem~\ref{thm:main:result:inf:dim}).
The main computational difficulty of the presented method lies in the gradient evaluation $\nabla F_{\eta}$. We refer to Section~\ref{sec:gradient:evalutaion}, for a detailed discussion on this subject.

\begin{myass}[Slater point] \label{ass:slater}
There exits a strictly feasible solution to \eqref{eq:main:problem}, i.e., $\mu_0\in\mathcal{P}(\K)$ such that $\mathcal{A}\mu_0\in T$ and $\delta:=\min_{y\in  T^c} \| \mathcal{A}\mu_0 - y\|_2 >0$.
\end{myass}
Note that finding a Slater point $\mu_0$ such that Assumption~\ref{ass:slater} holds, in general can be difficult. In Remark~\ref{rem:slater:point} we present a constructive way of finding such an interior point. Given Assumption~\ref{ass:slater}, for $\varepsilon>0$ define
\begin{align}
 &C:=\KL{\mu_0}{\nu}, \qquad D :={1 \over 2} \max_{x \in T} \|x\|_2, \qquad \eta_{1}(\varepsilon) :=\frac{\varepsilon}{4D},  \qquad  \eta_{2}(\varepsilon) :=\frac{\varepsilon \delta^2}{2C^2} \nonumber \\
 &N_1(\varepsilon):=2 \left( \sqrt{\frac{8DC^2}{\varepsilon^2 \delta^2}+\frac{2\|\mathcal{A}\|^2 C^2}{\varepsilon \delta^2}+1}\right) \ln\left(\frac{10(\varepsilon +2C)}{\varepsilon}\right) \label{eq:definitions:algo:cont} \\
&N_2(\varepsilon):=2 \left(\! \sqrt{\frac{8DC^2}{\varepsilon^2 \delta^2}+\frac{2\|\mathcal{A}\|^2 C^2}{\varepsilon \delta^2}+1}\!\right)\! \ln\!\left(\! \frac{C}{\varepsilon \delta(2-\sqrt{3})}\sqrt{4\left(\! \frac{4D}{\varepsilon}+\|\mathcal{A}\|^2 + \frac{\varepsilon \delta^2}{2C^2}\! \right)\!\left(\! C +\frac{\varepsilon}{2} \right)}\! \right). \nonumber
\end{align}

\begin{theorem}[Convergence rate]\label{thm:main:result:inf:dim}
Given Assumption~\ref{ass:slater} and~\eqref{eq:definitions:algo:cont}, let $\varepsilon>0$ and $N(\varepsilon) := \left \lceil \max\{ N_1(\varepsilon), N_2(\varepsilon)\} \right \rceil$. Then, $N(\varepsilon)$ iterations of Algorithm~\hyperlink{algo:1}{1} produce approximate solutions to the problems
\eqref{eq:dual:problem} and \eqref{eq:primal:problem} given by
\begin{equation} \label{eq:estimates:primal:dual}
\hat{z}_{k,\eta}:=y_{k} \quad \text{and} \quad \hat{\mu}_{k,\eta}(B) := \frac{\int_B 2^{-\mathcal{A}^*\hat{z}_{k,\eta}(x)}\nu(\drv x)}{\int_\K 2^{-\mathcal{A}^*\hat{z}_{k,\eta}(x)}\nu(\drv x)}\, , \quad \text{for all } B\in\mathcal{B}(\K) \, ,
\end{equation}
which satisfy
\begin{subequations}
\begin{alignat}{3}
&\text{dual $\varepsilon$-optimality:}\hspace{20mm} &&0\leq J^{\star} - F(\hat{z}_{k(\varepsilon)})\leq \varepsilon \label{eq:thm:dual:optimality:cts} \\
&\text{primal $\varepsilon$-optimality:}\hspace{25mm} &&|\KL{\hat{\mu}_{k(\varepsilon)}}{\nu}- J^{\star} | \leq  2(1+2\sqrt{3})\varepsilon \label{eq:thm:primal:optimality:cts} \\
&\text{primal $\varepsilon$-feasibility:}\hspace{20mm} &&d(\mathcal{A}\hat{\mu}_{k(\varepsilon)},T)\leq \frac{2\varepsilon\delta}{C} \, , \label{eq:thm:primal:feasibility:cts}
\end{alignat}
\end{subequations}
where $d(\cdot,T)$ denotes the distance to the set $T$, i.e., $d(x,T):=\min_{y\in T}\|x-y\|_2$.
\end{theorem}

In some applications, Assumption~\ref{ass:slater} does not hold, as for example in the classical case where $\mathcal{U}_i=\{0\}$ for all $i=1,\ldots,M$, i.e., there is no uncertainty in the measurement model~\eqref{eq:measurement:model}. Moreover, in other cases satisfying Assumption~\ref{ass:slater} using the construction described in Remark~\ref{rem:slater:point} might be computationally expensive. 
Interestingly, Algorithm~\hyperlink{algo:1}{1} can be run irrespective of whether Assumption~\ref{ass:slater} holds or not, i.e. for any choice of $C$ and $\delta$. While explicit error bounds of Theorem~\ref{thm:main:result:inf:dim} as well as the a-posteriori error bound discussed below do not hold anymore, the asymptotic convergence is not affected.

\begin{proof}
Using Assumption~\ref{ass:slater}, note that the constant defined as
\begin{equation*}
\iota:= \frac{\KL{\mu_0}{\nu}-\min_{\mu\in\mathcal{P}(\K)}\KL{\mu}{\nu}}{\min_{y\in  T^c} \| \mathcal{A}\mu_0 - y\|_2} = \frac{C}{\delta}
\end{equation*}
can be shown to be an upper bound for the optimal dual multiplier \cite[Lemma 1]{ref:Nedic-08}, i.e., $\| z^\star \|_2 \leq \iota$.
The dual function can be bounded from above by $C$, since weak duality ensures $ F(z) \leq J^\star \leq \KL{\mu_0}{\nu}=C$ for all $z \in \R^M$.
Moreover, if we recall the preparatory Lemmas~\ref{lem:zero:duality:gap} and \ref{lem:lipschitz:cts:gradient}, we are finally in the setting such that the presented error bounds can be derived from \cite{ref:devolder-12}, see Appendix~\ref{app:detailed:proof} for a detailed derivation.
\end{proof}

Theorem~\ref{thm:main:result:inf:dim} directly implies that we need at most $O(\frac{1}{\varepsilon} \log \frac{1}{\varepsilon})$ iterations of Algorithm~\hyperlink{algo:1}{1} to achieve $\varepsilon$-optimality of primal and dual solutions as well as $\varepsilon$-feasible primal variable. 
Note that Theorem~\ref{thm:main:result:inf:dim} provides an explicit bound on the so-called \emph{a-priori errors}, together with approximate optimizer of the primal \eqref{eq:primal:problem} and dual \eqref{eq:dual:problem} problem. The latter allows us to derive an \emph{a-posteriori error} depending on the approximate optimizers, which is often significantly smaller than the a-priori error.
\begin{corollary}[Posterior error estimation] \label{Cor:aposteriori:infinite:dim}
Given Assumption~\ref{ass:slater}, the approximate primal and dual variables $\hat{\mu}$ and $\hat{z}$ given by~\eqref{eq:estimates:primal:dual}, satisfy the following a-posteriori error bound 
\begin{align*}
F(\hat{z})\leq J^{\star}\leq \KL{\hat{\mu}}{\nu} + \frac{C}{\delta} d(\mathcal{A}\hat{\mu},T) \, ,
\end{align*}
where $d(\cdot,T)$ denotes the distance to the set $T$, i.e., $d(x,T):=\inf_{y\in T}\|x-y\|_2$.
\end{corollary}
\begin{proof}
The two key ingredients of the proof are Theorem~\ref{thm:main:result:inf:dim} and the Lipschitz continuity of the so-called perturbation function of convex programming. Let $z^{\star}$ denote the dual optimizer to \eqref{eq:dual:problem}. We introduce the perturbed program as
\begin{subequations}
 \begin{align} 
 	       \quad J^{\star}(\varepsilon) &= \min\limits_{\mu\in\mathcal{P}(\K)} \{ \KL{\mu}{\nu}  \, : \, d(\mathcal{A}\mu,T)\leq \varepsilon  \} \nonumber\\
		&= \min\limits_{\mu\in\mathcal{P}(\K)}  \KL{\mu}{\nu} + \sup\limits_{\lambda\geq 0} \inf_{y\in T} \lambda \| \mathcal{A}\mu-y \| - \lambda \varepsilon \nonumber\\
		&= \sup\limits_{\lambda\geq 0} - \lambda \, \varepsilon + \inf_{\substack{\mu\in\mathcal{P}(\K)\\  y\in T}} \sup_{\| z \|_{2}\leq \lambda} \inprod{\mathcal{A}\mu-y}{z} + \KL{\mu}{\nu} \label{eq:cor:step:slater:cts} \\
		&=  \sup_{\substack{\lambda\geq 0 \\ \| z \|_{2}\leq \lambda}} - \lambda\, \varepsilon + \inf_{\substack{\mu\in\mathcal{P}(\K)\\  y\in T}} \inprod{\mathcal{A}\mu-y}{z} + \KL{\mu}{\nu}\label{eq:cor:step:sion:cts}\\
		&\geq  -\| z^{\star} \|_{2}\, \varepsilon + \inf_{\substack{\mu\in\mathcal{P}(\K)\\  y\in T}} \inprod{\mathcal{A}\mu-y}{z^{\star}} + \KL{\mu}{\nu} \nonumber \\
		& = -\| z^{\star} \|_{2} \, \varepsilon + J^{\star}.  \nonumber
 	\end{align}
	\end{subequations}
Equation~\eqref{eq:cor:step:slater:cts} uses the strong duality property that follows by the existence of a Slater point that is due to the definition of the set $T$, see Section~\ref{sec:problem:statement}. Step \eqref{eq:cor:step:sion:cts} follows by Sion's minimax theorem~\cite[Theorem~4.2]{ref:Sion-58}. Hence, we have shown that the perturbation function is Lipschitz continuous with constant $\| z^{\star} \|_{2}$. Finally, recalling $\| z^\star \|_2 \leq \frac{C}{\delta}$, established in the proof of Theorem~\ref{thm:main:result:inf:dim} completes the proof.
\end{proof}

\begin{remark}[Stopping criterion of Algorithm~\hyperlink{algo:1}{1}] \label{remark:stopping}
There are two alternatives for defining a stopping criterion for Algorithm~\hyperlink{algo:1}{1}. Choose desired accuracy $\varepsilon>0$.
\begin{enumerate}[(i)]
\item \emph{a-priori stopping criterion}: Theorem~\ref{thm:main:result:inf:dim} provides the required number of iterations $N(\varepsilon)$ to ensure an $\varepsilon$-close solution.
\item \emph{a-posteriori stopping criterion}: Choose the smoothing parameter $\eta$ as in \eqref{eq:definitions:algo:cont}. Fix a (small) number of iterations $\ell$ that are run using Algorithm~\hyperlink{algo:1}{1}. Compute the a-posteriori error $\KL{\hat{\mu}}{\nu} + \frac{C}{\delta} d(\mathcal{A}\hat{\mu},T)  - F(\hat{z})$ according to Corollary~\ref{Cor:aposteriori:infinite:dim} and if it is smaller than If $\varepsilon$ terminate the algorithm. Otherwise continue with another $\ell$ iterations.
\end{enumerate}
\end{remark}

\begin{remark}[Slater point computation] \label{rem:slater:point}
To compute the respective constants in Assumption~\ref{ass:slater}, we need to construct a strictly feasible point for \eqref{eq:main:problem}. For this purpose, we consider a polynomial density of degree $r$ defined as $p_r(\alpha,x) := \sum_{i=0}^{r-1} \alpha_i x^i$. For notational simplicity we assume that the support set is the unit interval ($\K = [0,1]$), such that the moments induced by the polynomial density are given by
\begin{align*}
\inprod{p_r(\alpha,x)}{x^i} = \int_0^1 \sum_{j=0}^{r-1} \alpha_j x^{j+i} \drv x= \sum_{j=0}^{r-1} \frac{\alpha_j}{j+i+1},
\end{align*}
for $i=0,\hdots, M$. Consider $\beta\in\R^{M+1}$, where $\beta_1 = 1$ and $\beta_i=y_{i-1}$ for $i=2,\hdots,M+1$. Hence, the feasibility requirement of \eqref{eq:main:problem} can be expressed as the linear constraint $A \alpha = \beta$,
where $A\in\R^{(M+1)\times r}$, $\alpha\in\R^r$, $\beta\in\R^{M+1}$ and $A_{i,j} = \frac{1}{i+j-1}$
and finding a strictly feasible solution reduces to the following feasibility problem
\begin{align} \label{primal-n} 
 \left\{ \begin{array}{ll}
		\max\limits_{\alpha \in \R^r} & \text{const}   \\
		\st & A\alpha = \beta \\
		& p_r(\alpha,x) \geq 0 \quad \forall x\in[0,1], 
\end{array} \right.
\end{align}
where $p_r$ is a polynomial function in $x$ of degree $r$ with coefficients $\alpha$. 
The second constraint of the program \eqref{primal-n} (i.e., $p_r(\alpha,x) \geq 0 \ \forall x\in[0,1]$)\footnote{In a multi-dimensional setting one has to consider a tightening (i.e., $p_r(\alpha,x) >0 \ \forall x\in[0,1]^n$).} can be equivalently reformulated as linear matrix inequalities of dimension $\ceil{\frac{r}{2}}$, using a standard result in polynomial optimization, see \cite[Chapter~2]{ref:Lasserre-11} for details.
We note that for small degree $r$, the set of feasible solutions to problem \eqref{primal-n} may be empty, however, by choosing $r$ large enough and assuming that the moments can be induced by a continuous density, problem \eqref{primal-n} becomes feasible. Moreover, if $0 \in \mathsf{int}(T)$ the Slater point leads to a $\delta>0$ in Assumption~\ref{ass:slater}.
\end{remark}

\begin{example}[Density estimation] \label{ex:Example1}
We are given the first $3$ moments of an unknown probability measure defined on $\K=[0,1]$ as\footnote{The considered moments are actually induced by the probability density $f(x):=(\ln 2 \ (1+x))^{-1}$. We, however, do not use this information at any point of this example.}
\begin{align*}
y:=\left( \frac{1-\ln 2}{\ln 2}, \frac{\ln 4 - 1}{\ln 4}, \frac{5-\ln 64}{\ln 64} \right) \approx (0.44,\ 0.28,\ 0.20).
\end{align*}
The uncertainty set in the measurement model \eqref{eq:measurement:model} is assumed to be $\mathcal{U}_{i}=[-u,u]$ for all $i=1,\hdots,3$. A Slater point is constructed using the method described in Remark~\ref{rem:slater:point}, where $r=5$ is enough for the problem~\eqref{primal-n} to be feasible, leading to the constant $C=0.0288$. The Slater point is depicted in Figure~\ref{fig:plotExperimentMeanVar1} and its differential entropy can be numerically computed as $-0.0288$.

We consider two  simulations for two different uncertainty sets (namely, $u=0.01$ and $u=0.005$). 
The underlying maximum entropy problem \eqref{eq:main:problem:special case} is solved using Algorithm~\hyperlink{algo:1}{1}. The respective features of the a-priori guarantees by Theorem~\ref{thm:main:result:inf:dim} as well as the a-posteriori guarantees (upper and lower bounds) by Corollary~\ref{Cor:aposteriori:infinite:dim} are reported in Table~\ref{tab:ex1}. 
Recall that $\hat{\mu}_{k(\varepsilon)}$ denotes the approximate primal variable after $k$-iterations of Algorithm~\hyperlink{algo:1}{1} as defined in Theorem~\ref{thm:main:result:inf:dim} and that $d(\mathcal{A}\hat{\mu}_{k(\varepsilon)},T)$ (resp.~$\frac{2\varepsilon \delta}{C}$) represent the a-posteriori (resp.~a-priori) feasibility guarantees.
It can be seen in Table~\ref{tab:ex1} that increasing the uncertainty set $\mathcal{U}$ leads to a higher entropy, where the uniform density clearly has the highest entropy. This is also intuitively expected since enlarging the uncertainty set is equivalent to relaxing the moment constraints in the respective maximum entropy problem. The corresponding densities are graphically visualized in Figure~\ref{fig:plotExperimentMeanVar1}.

 \begin{table}[!htb]
\centering 
\caption{Some specific simulation points of Example~\ref{ex:Example1}. }
\label{tab:ex1}
\hspace{40mm} $\mathcal{U}=[-0.01,0.01]$ \hspace{29mm} $\mathcal{U}=[-0.005,0.005]$
\vspace{3mm} \phantom{..}
  \begin{tabular}{c@{\hskip 1.5mm} | c@{\hskip 1.5mm} c@{\hskip 1.5mm} c@{\hskip 1.5mm} c  | c@{\hskip 1.5mm} c@{\hskip 1.5mm} c@{\hskip 1.5mm} c  }
 a-priori error $\varepsilon$ \hspace{1mm}  & \hspace{1mm}    1  &$0.1$ & $0.01$ & $0.001$ \hspace{1mm}   &\hspace{1mm}  1  &$0.1$ & $0.01$ & $0.001$   \\ 
 $J_{\textnormal{UB}}$ & \hspace{1mm}-0.0174 & -0.0189 & -0.0194  & -0.0194 \hspace{1mm}   &\hspace{1mm}  -0.0223 & -0.0236 & -0.0237 & -0.0238   \\
 $J_{\textnormal{LB}}$ & \hspace{1mm} -0.0220 & -0.0279 & -0.0204  & -0.0195 \hspace{1mm}   & \hspace{1mm}   -0.0263 & -0.0298 & -0.0244 & -0.0238 \\
 iterations $k(\varepsilon)$ & \hspace{1mm} 99 & 551 & 5606  & 74423 \hspace{1mm}  & \hspace{1mm}   232 & 1241 & 12170 & 157865 \\
 $d(\mathcal{A}\hat{\mu}_{k(\varepsilon)},T)$& \hspace{1mm} 0.0008  &0.0036 &0.0005  & 0 \hspace{1mm} & \hspace{1mm}  0 & 0.001 & 0.0001 & 0  \\
  $\frac{2\varepsilon \delta}{C}$ & \hspace{1mm} 0.69  &0.069 &0.0069  & 0.00069 \hspace{1mm} & \hspace{1mm}  0.35 & 0.035 & 0.0035 & 0.00035 \\
   runtime [s]\tablefootnote{Runtime includes Slater point computation. Simulations were run with Matlab on a laptop with a 2.2 GHz Intel Core i7 processor.} & \hspace{1mm}  1.4 &1.4 & 2.3 & 12.9 \hspace{1mm} & \hspace{1mm}1.4 & 1.5 & 3.3 & 26.1 
  \end{tabular}
\end{table}

\begin{figure}[!htb]   
\centering
    {\input{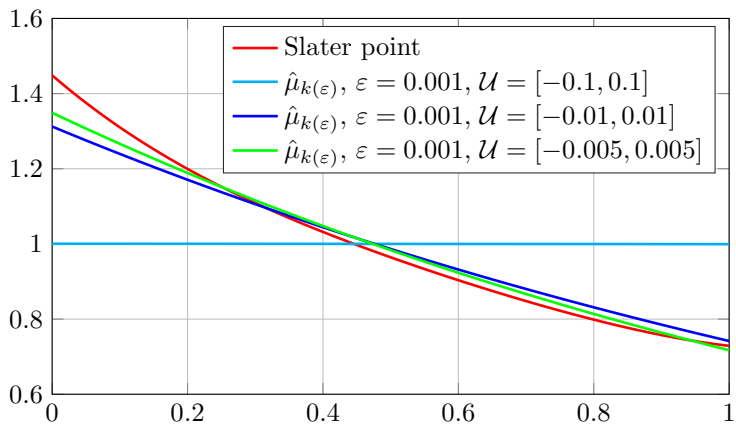} \label{fig:firstExperiment}}
    \caption[]{Maximum entropy densities obtained by Algorithm~\hyperlink{algo:1}{1} for two different uncertainty sets. As a reference, the Slater point density, that was computed as described in Remark~\ref{rem:slater:point} is depicted in red.}
    \label{fig:plotExperimentMeanVar1}
\end{figure}

%

%
%
%
%

\end{example}

\section{Finite-dimensional case} \label{sec:finite:dim:case}
We consider the finite-dimensional case where $\K=\{1,\hdots,N\}$ and hence we optimize in \eqref{eq:main:problem} over the probability simplex $\mathcal{P}(\K)=\Delta_{N}$. One substantial simplification, when restricting to the finite-dimensional setting, is that the Shannon entropy is non-negative and bounded from above (by $\log N$). 
Therefore, we can substantially weaken Assumption~\ref{ass:slater} to the following assumption.

\begin{myass}[Regularity] \label{ass:finite:dim} \
\begin{enumerate}[(i)]
\item There exists $\delta>0$ such that $B(0,\delta)\subset\{x-\mathcal{A}\mu \ : \ \mu\in\Delta_{N}, x\in T\}$.
\item The reference measure $\nu\in\Delta_{N}$ has full support, i.e., $\min\limits_{1\leq i \leq N} \nu_{i}>0$.
\end{enumerate}
\end{myass}

Consider the the definitions given in \eqref{eq:definitions:algo:cont} with $C:=\max\limits_{1\leq i \leq N} \log \frac{1}{ \nu_{i}}$, then the following finite-dimensional equivalent to Theorem~\ref{thm:main:result:inf:dim} holds.

\begin{corollary}[A-priori error bound]\label{cor:main:result:finite:dim}
Given Assumption~\ref{ass:finite:dim}, $C:=\max\limits_{1\leq i \leq N} \log \frac{1}{ \nu_{i}}$ and the definitions~\eqref{eq:definitions:algo:cont}, let $\varepsilon>0$ and $N(\varepsilon) := \left \lceil \{ N_1(\varepsilon) , N_2(\varepsilon) \} \right \rceil$. Then, $N(\varepsilon)$ iterations of Algorithm~\hyperlink{algo:1}{1} produce the approximate solutions to the problems
\eqref{eq:dual:problem} and \eqref{eq:primal:problem}, given by
\begin{equation} \label{eq_varAlgo}
\hat{z}_{k(\varepsilon)}:=y_{k(\varepsilon)} \quad \text{and} \quad \hat{\mu}_{k(\varepsilon)}(B) := \frac{\sum_{i\in B}2^{-\left(\mathcal{A}^*\hat{z}_{k(\varepsilon)}\right)_i}\nu_i}{\sum_{i=1}^N 2^{-\left(\mathcal{A}^*\hat{z}_{k(\varepsilon)}\right)_i}\nu_i} \quad \text{for all }B\subset \{1,2,\hdots,N\} \, ,
\end{equation}
which satisfy
\begin{subequations}
\begin{alignat}{3}
&\text{dual $\varepsilon$-optimality:}\hspace{20mm} &&0\leq F(\hat{z}_{k(\varepsilon)})-J^{\star}\leq \varepsilon \label{eq:thm:dual:optimality} \\
&\text{primal $\varepsilon$-optimality:}\hspace{25mm} &&|\KL{\hat{\mu}_{k(\varepsilon)}}{\nu}- J^{\star} | \leq  2(1+2\sqrt{3})\varepsilon \label{eq:thm:primal:optimality} \\
&\text{primal $\varepsilon$-feasibility:}\hspace{20mm} &&d(\mathcal{A}\hat{\mu}_{k(\varepsilon)},T)\leq \frac{2\varepsilon\delta}{C} \, , \label{eq:thm:primal:feasibility}
\end{alignat}
\end{subequations}
where $d(\cdot,T)$ denotes the distance to the set $T$, i.e., $d(x,T):=\min_{y\in T}\|x-y\|_2$.
\end{corollary}
\begin{proof}
Under Assumption~\ref{ass:finite:dim} the dual optimal solutions in \eqref{eq:dual:problem} are bounded by 
\begin{align} \label{eq:dual:var:bound}
\| z^{\star} \| \leq \frac{1}{r}\max\limits_{1\leq i \leq N} \log \frac{1}{ \nu_{i}} \, .
\end{align}
This bound on the dual optimizer follows along the lines of~\cite[Theorem~6.1]{ref:devolder-12}. The presented error bounds can then be derived along the lines of Theorem~\ref{thm:main:result:inf:dim}.
\end{proof}
In addition to the explicit error bound provided by Corollary~\ref{cor:main:result:finite:dim}, the a-posteriori upper and lower bounds presented in Corollary~\ref{Cor:aposteriori:infinite:dim} directly apply to the finite-dimensional setting as well.

\section{Gradient Approximation} \label{sec:gradient:evalutaion}
The computationally demanding element for Algorithm~\hyperlink{algo:1}{1} is the evaluation of the gradient $\nabla F_{\eta}(\cdot)$ given in Lemma~\ref{lem:lipschitz:cts:gradient}. In particular, Theorem~\ref{thm:main:result:inf:dim} and Corollary~\ref{cor:main:result:finite:dim} assume that this gradient is known exactly. While this is not restrictive if, for example, $\K$ is a finite set, in general, $\nabla F_{\eta}(\cdot)$ involves an integration that can only be computed approximately. In particular if we consider a multi-dimensional setting (i.e., $\K\subset \R^d$), the evaluation of the gradient $\nabla F_{\eta}(\cdot)$ represents a multi-dimensional integration problem. This gives rise to the question of how the fast gradient method (and also Theorem~\ref{thm:main:result:inf:dim}) behaves in a case of inexact first-order information. Roughly speaking, the fast gradient method Algorithm~\hyperlink{algo:1}{1}, while being more efficient than the classical gradient method (if applicable), is less robust when dealing with inexact gradients \cite{ref:Devolder-13}. Therefore, depending on the computational complexity of the gradient, one may consider the possibility of replacing Algorithm~\hyperlink{algo:1}{1} with a classical gradient method. A detailed mathematical analysis of this tradeoff is a topic of further research, and we refer the interested readers to \cite{ref:Devolder-13} for further details in this regard.

In this section we discuss two numerical methods to approximate this gradient. Note that in Lemma~\ref{lem:lipschitz:cts:gradient}, given that $T$ is simple enough the optimizer $x^{\star}_{z}$ is analytically available, so what remains is to compute $\mathcal{A}\mu^{\star}_{z}$, that according to Lemma~\ref{lem:entropy:max} is given by
\begin{align} \label{eq:gradient:approximation}
(\mathcal{A}\mu^{\star}_{z})_{i} &= \frac{\int_{\K} x^{i}2^{-\mathcal{A}^{*}z(x)}\nu(\drv x)}{\int_{\K} 2^{-\mathcal{A}^{*}z(x)}\nu(\drv x)} \quad\text{ for all }i=1,\hdots, M \, .
\end{align}

%
%
%
%

\textbf{Semidefinite programming.} Due to the specific structure of the considered problem, \eqref{eq:gradient:approximation} represents an integration of exponentials of polynomials for which an efficient approximation in terms of two single semidefinite programs (SDPs) involving two linear matrix inequalities has been derived, where the desired accuracy is controlled by the size of the linear matrix inequalities constraints, see~\cite{ref:Bertsimas-08, ref:Lasserre-11} for a comprehensive study and for the construction of those SDPs.
While the mentioned hierarchy of SDPs provides a certificate of optimality (hat is easy to evaluate and asymptotic convergence (in the size of the SDPs), a convergence rate that explicitly quantifies the size of the SDPs required for a desired accuracy is unknown. In practice, the hierarchy often converges in few iteration steps, which however, depends on the problem and is not known a priori.

\textbf{Quasi-Monte Carlo.}
The most popular methods for integration problems of the from \eqref{eq:gradient:approximation} are Monte Caro (MC) schemes, see \cite{ref:RobertCasella-04} for a comprehensive summary. The main advantage of MC methods is that the root-mean-square error of the approximation converges to $0$ with a rate of $O(N^{-1/2})$ that is independent of the dimension, where $N$ are the number of samples used.
In practise, this convergence often is too slow. Under mild assumptions on the integrand, the MC methods can be significantly improved with a more recent technique known as Quasi-Monte Carlo (QMC) methods. QMC methods can reach a convergence rate arbitrarily close to $O(N^{-1})$ with a constant not depending on the dimension of the problem. We would like to refer the reader to \cite{ref:Dick-13, Sloan-98, Sloan-05, niederreiter2010quasi, MT_Maxime} for a detailed discussion about the theory of QMC methods.

\begin{remark}[Computational stability] \label{rem:computational:stability}
The evaluation of the gradient in Lemma~\ref{lem:lipschitz:cts:gradient} involves the term $\mathcal{A}\mu^{\star}_{z}$, where $\mu^{\star}_{z}$ is the optimizer of the second term in \eqref{eq:dual:function:smoothed}. By invoking Lemma~\ref{lem:entropy:max} and the definition of the operator $\mathcal{A}$, the gradient evaluation reduces to
\begin{equation} \label{eq:rem:stable}
\left(\mathcal{A}\mu^{\star}_{z} \right)_i = \frac{\int_\K x^i 2^{-\sum_{j=1}^M z_j x^j}\drv x}{\int_\K 2^{-\sum_{j=1}^M z_j x^j}\drv x} \quad \text{for }i=1,\hdots,M \, .
\end{equation}
Note that a straightforward computation of the gradient via \eqref{eq:rem:stable} is numerically difficult. To alleviate this difficulty, we follow the suggestion of \cite[p.~148]{nesterov05} which we briefly elaborate here. Consider the functions $f(z,x):= -\sum_{j=1}^M z_j x^j $, $\bar{f}(z):=\max_{x\in\K} f(z,x)$ and $g(z,x):= f(z,x)- \bar{f}(z)$. Note, that  $g(z,x)\geq 0$ for all $(z,x)\in \R^M\times\R$. One can show that
\begin{equation*} 
\left(\mathcal{A}\mu^{\star}_{z} \right)_i = \frac{\int_\K  2^{g(z,x)} \frac{\partial}{\partial z_i}g(z,x)\drv x}{\int_\K 2^{g(z,x)}\drv x} + \frac{\partial}{\partial z_i}\bar{f}(z) \quad \text{for }i=1,\hdots,M \, ,
\end{equation*}
which can be computed with significantly smaller numerical error than \eqref{eq:rem:stable} as the numerical exponent are always negative, leading to values always being smaller than $1$. 
\end{remark}

\section{Zero-information moment closure method} \label{sec:moment:closure}

In chemistry, physics, systems biology and related fields, stochastic chemical reactions are described by the \emph{chemical master equation} (CME), that is a special case of the Chapman-Kolmogorov equation as applied to Markov processes~\cite{ref:vanKampen-81, ref:Wilkinson-06}.
These equations are usually infinite-dimensional and analytical solutions are generally impossible. Hence, effort has been directed toward developing of a variety of numerical schemes for efficient approximation of the CME, such as stochastic simulation techniques (SSA)~\cite{ref:Gillespie-76}. In practical cases, one is often interested in the first few moments of the number of molecules involved in the chemical reactions. This motivated the development of approximation methods to those low-order moments without having to solve the underlying infinite-dimensional CME. One such approximation method is the so-called \emph{moment closure method}~\cite{ref:Gillespie-09}, that briefly described works as follows: First the CME is recast in terms of moments as a linear ODE of the form
\begin{equation} \label{eq:moment:CME}
\frac{\drv}{\drv t}\mu(t) =  A \mu(t) + B \zeta(t)\, ,
\end{equation}
where $\mu(t)$ denotes the moments up to order $M$ at time $t$ and $\zeta(t)$ is an infinite vector describing the contains moments of order $M+1$ or higher. 
In general $\zeta$ can be an infinite vector, but for most of the standard chemical reactions considered in, e.g., systems biology it turns out that only a finite number of higher order moments affect the evolution of the first $M$ moments. Indeed, if the chemical system involves only the so-called zeroth and first order reactions the vector $\zeta$ has dimension zero (reduces to a constant affine term), whereas if the system also involves second order reactions then $\zeta$ also contains some moments of order $M+1$ only. It is widely speculated that reactions up to second order are sufficient to realistically model most systems of interest in chemistry and biology \cite{ref:Gillespie-07, ref:Gillespie-13}.
The matrix $A$ and the linear operator $B$ (that may potentially be infinite-dimensional) can be found analytically from the CME. The ODE \eqref{eq:moment:CME}, however, is intractable due to its higher order moments dependence. The approximation step is introduced by a so-called closure function
\begin{equation*}
\zeta=\varphi(\mu)\, ,
\end{equation*}
where the higher-order moments are approximated as a function of the lower-oder moments, see \cite{ref:Singh-06,ref:Singh-07}. A closure function that has recently attracted interest is known as the \emph{zero-information} closure function (of order $M$)~\cite{ref:Smadbeck-13}, and is given by
\begin{equation} \label{eq:closure:function}
(\varphi(\mu))_i = \inprod{p^\star_\mu}{x^{M+i}}, \text{ for } i=1,2,\hdots\, ,
\end{equation}
where $p^\star_\mu$ denotes the maximizer to the problem \eqref{eq:main:problem:special case}, where $T=\times_{i=1}^M \left\{ \inprod{\mu}{x^i}-u \ : \ u\in\mathcal{U}_i \right\}$, where $\mathcal{U}_i=[-\kappa, \kappa]\subset\R$ for all $i$ and for a given $\kappa>0$, that acts as a regularizer, in the sense of Assumption~\ref{ass:slater}.
This approximation reduces the infinite-dimensional ODE \eqref{eq:moment:CME} to a finite-dimensional ODE
\begin{equation} \label{eq:finite:ODE}
\frac{\drv}{\drv t}\mu(t) =  A \mu(t) + B \varphi(\mu(t))\, .
\end{equation}

To numerically solve \eqref{eq:finite:ODE} it is crucial to have an efficient evaluation of the closure function $\varphi$. In the zero-information closure scheme this is given by to the entropy maximization problem \eqref{eq:main:problem:special case} and as such can be addressed using Algorithm~\hyperlink{algo:1}{1}.


\begin{figure}[!htb]
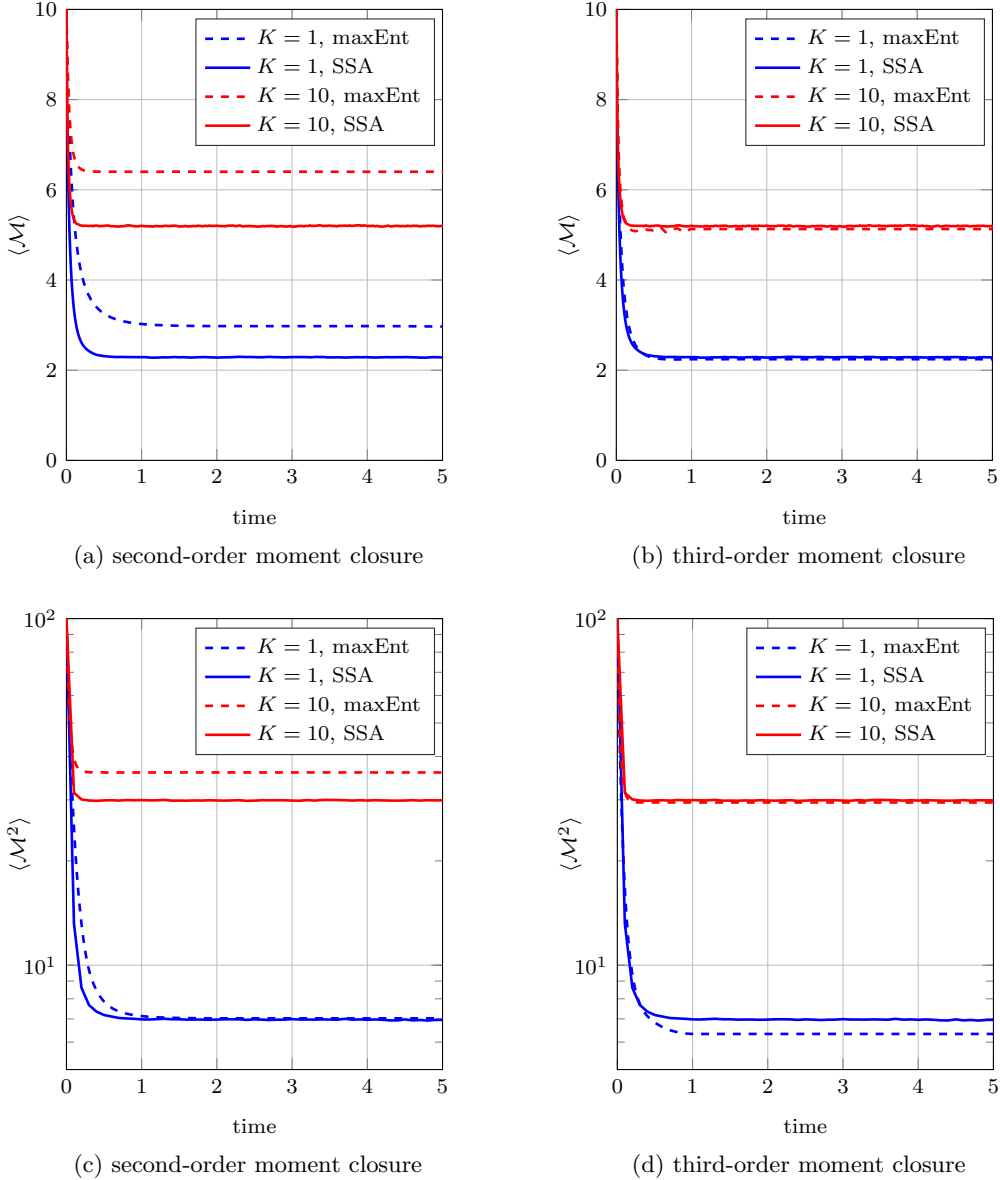
 
\centering
    \subfloat[second-order moment closure]{\input{secondOrder.tex} \label{fig:secondOrder} } \hspace{10mm}
    \subfloat[third-order moment closure]{\input{thirdOrder.tex} \label{fig:thirdOrder} } \\
     \subfloat[second-order moment closure]{\input{secondOrder_secondMoment.tex} \label{fig:secondOrder:secondMoment} } \hspace{10mm}
    \subfloat[third-order moment closure]{
%
%
\begin{tikzpicture}

\begin{axis}[%
height=6.0cm,
width=5cm,
at={(1.011111in,0.641667in)},
scale only axis,
grid=major,
style={font=\scriptsize},
y label style={at={(axis description cs:0.12,.5)}},
xmin=0,
xmax=5,
xlabel={time},
ylabel={$\langle \mathcal{M}^2\rangle$},
ymin=5,
ymax=100,
ymode=log,
xmode=linear,
legend style={legend cell align=left,align=left,draw=white!15!black}
]


\addplot [color=blue,dashed,line width=1.0pt]
  table[row sep=crcr]{%
0	100\\
0.00247348597181315	92.4303194890925\\
0.00735048582743701	79.7756597419882\\
0.0123278178962352	69.3763592898415\\
0.0174980692934942	60.6271824979529\\
0.0229267736098641	53.1635322315972\\
0.0286566963140782	46.7526157518448\\
0.0347325413593649	41.2158769059201\\
0.0412042415823096	36.4123321323829\\
0.0481269628495207	32.2302481477056\\
0.0555607717361404	28.5807325173315\\
0.0635701996395642	25.3925783149999\\
0.0722242536485956	22.6079613716804\\
0.0815976352908185	20.1788557170896\\
0.0917725906976913	18.0644495400862\\
0.102841854737132	16.2292809556737\\
0.114911882161327	14.642086532445\\
0.128107235703439	13.2749865111902\\
0.142576111877736	12.1029436598612\\
0.158497230331308	11.1033923303349\\
0.176088858004435	10.2559389107605\\
0.195620873305622	9.5420950034733\\
0.217431807256255	8.9450074848086\\
0.241955222520164	8.44916197827403\\
0.26976155435028	8.04010814873858\\
0.301621452972396	7.70427721753663\\
0.338620939864165	7.42876226411267\\
0.382348154600584	7.2005243259817\\
0.428873912780483	7.02327974189987\\
0.500593103370028	6.82173365550059\\
0.600593103370028	6.62647062575438\\
0.700593103370028	6.49392463738132\\
0.800593103370028	6.40525390121874\\
0.900593103370028	6.35267801701014\\
1	6.33463914589916\\
5       6.33463914589916\\
};
\addlegendentry{$K=1$, maxEnt};

\addplot [color=blue,solid,line width=1.0pt]
  table[row sep=crcr]{%
0	100\\
0.1	13.1846762110879\\
0.2	8.60518069868365\\
0.3	7.67705910115536\\
0.4	7.34545603269304\\
0.5	7.18353210621548\\
0.6	7.1114579395795\\
0.7	7.04235168650224\\
0.8	7.0229051606718\\
0.9	7.00530140166039\\
1	6.97863980223658\\
1.1	6.97242185713404\\
1.2	6.98942429059647\\
1.3	7.00643603998493\\
1.4	6.98334311510182\\
1.5	6.9831067546736\\
1.6	6.96556439090907\\
1.7	6.97227305231403\\
1.8	6.96918606625087\\
1.9	6.96020496539274\\
2	6.96809164733651\\
2.1	6.98224668882705\\
2.2	6.98342464797445\\
2.3	6.97864265030472\\
2.4	6.95555347531089\\
2.5	6.97917716800895\\
2.6	6.96680259550642\\
2.7	6.94904421478716\\
2.8	6.96503539423505\\
2.9	6.96706152631044\\
3	6.98958444613472\\
3.1	6.97143117396451\\
3.2	6.97142143465491\\
3.3	6.96786581964804\\
3.4	6.97759320887489\\
3.5	6.97023878683558\\
3.6	6.96366759608914\\
3.7	7.00368375292945\\
3.8	6.9650680764096\\
3.9	6.97689139280527\\
4	6.96028149587453\\
4.1	6.97621009089989\\
4.2	6.96055792892175\\
4.3	6.93364512458152\\
4.4	6.93009790531627\\
4.5	6.97740664964208\\
4.6	6.93255333039614\\
4.7	6.95181937901597\\
4.8	6.92442013912225\\
4.9	6.95202987137574\\
5	6.95736862670111\\
};
\addlegendentry{$K=1$, SSA};


\addplot [color=red,dashed,line width=1.0pt, each nth point={1}]
  table[row sep=crcr]{%
0	100\\
0.0011438505922504	96.4104102370032\\
0.0022877011845008	93.0423386769572\\
0.0034315517767512	89.8795154069735\\
0.0045754023690016	86.9068855146559\\
0.00744852627070453	80.1839856789375\\
0.0103216501724075	74.3946946742465\\
0.0131947740741104	69.3956201377125\\
0.0160678979758133	65.0518823930898\\
0.0190932302350526	61.0645677574669\\
0.0221185624942919	57.5974037885953\\
0.0251438947535313	54.5747021757256\\
0.0281692270127706	51.9260084977718\\
0.0316646866544118	49.255049453673\\
0.0351601462960529	46.9447987254587\\
0.0386556059376941	44.941976506034\\
0.0421510655793353	43.1964153713399\\
0.046197509507085	41.4426925458071\\
0.0502439534348347	39.9364590486322\\
0.0542903973625844	38.6405954313525\\
0.0583368412903341	37.5196705256296\\
0.0631029909079174	36.384216430406\\
0.0678691405255006	35.4210358146173\\
0.0726352901430838	34.6033958239671\\
0.077401439760667	33.9052810012997\\
0.0831080943944591	33.1968366625455\\
0.0888147490282511	32.6074292233886\\
0.0945214036620432	32.1174968954863\\
0.100228058295835	31.7074910374806\\
0.107181820724837	31.2929955383042\\
0.11413558315384	30.9580372353421\\
0.121089345582842	30.6884311524954\\
0.128043108011844	30.469408154333\\
0.136665685200106	30.250488960823\\
0.145288262388369	30.080930328584\\
0.153910839576631	29.9510699148845\\
0.162533416764893	29.8500049502439\\
0.173388410595362	29.7507238450427\\
0.184243404425831	29.6783746274986\\
0.1950983982563	29.6273947477368\\
0.205953392086768	29.5901018363626\\
0.219750911047443	29.5534235192299\\
0.233548430008118	29.5289773619351\\
0.247345948968793	29.514723606201\\
0.261143467929468	29.5052536185119\\
0.278644351428597	29.4934069119628\\
0.296145234927726	29.4868140478167\\
0.313646118426856	29.4865358204288\\
0.331147001925985	29.486484846078\\
0.351557031744073	29.4716457773141\\
0.371967061562161	29.4672731742283\\
0.392377091380249	29.4862527798329\\
0.412787121198337	29.4958719712177\\
0.424241845468213	29.4858839181489\\
0.435696569738089	29.4810328676897\\
0.447151294007965	29.4830917618974\\
0.458606018277841	29.4849538792234\\
0.470060742547717	29.4820492764798\\
0.481515466817593	29.4806441808781\\
0.492970191087469	29.4812645321659\\
0.504424915357345	29.4818204746261\\
0.519939489865311	29.4795093665395\\
0.535454064373278	29.4787093408972\\
0.550968638881244	29.4810387088347\\
0.566483213389211	29.4824864487988\\
0.58534750253132	29.4727475764936\\
0.604211791673428	29.4700747938641\\
0.623076080815537	29.4838215674651\\
0.641940369957646	29.4911751254512\\
0.654115550667265	29.4828262442801\\
0.666290731376883	29.4790595585703\\
0.678465912086502	29.4821714623455\\
0.690641092796121	29.4846203614801\\
0.702816273505739	29.4813957191312\\
0.714991454215358	29.479944494842\\
0.727166634924977	29.4811623568749\\
0.739341815634595	29.4821176520533\\
0.754780591940258	29.4792154281429\\
0.77021936824592	29.4782141238194\\
0.785658144551582	29.4811466934853\\
0.801096920857244	29.4829763028976\\
0.81902212980442	29.4737082937901\\
0.836947338751596	29.471011859682\\
0.854872547698772	29.4833243280892\\
0.872797756645947	29.4901874653486\\
0.885280210502631	29.4820795618251\\
0.897762664359314	29.4785149525111\\
0.910245118215998	29.4820542862879\\
0.922727572072681	29.4847436395811\\
0.935210025929365	29.4811887241214\\
0.947692479786048	29.4796294255005\\
0.960174933642731	29.4811977025947\\
0.972657387499415	29.4823862489473\\
0.979493040624561	29.4817665223727\\
0.986328693749707	29.4813516523387\\
0.993164346874854	29.4810984095674\\
1	29.480931170018\\
5	29.480931170018\\
};  
\addlegendentry{$K=10$, maxEnt};

\addplot [color=red,solid,line width=1.0pt]
  table[row sep=crcr]{%
0	100\\
0.1	31.4360922253912\\
0.2	30.0725916160037\\
0.3	29.867174510949\\
0.4	29.8085674534831\\
0.5	29.8875101161383\\
0.6	29.9171280835894\\
0.7	29.8554049546653\\
0.8	29.961318657568\\
0.9	29.9114373067155\\
1	29.9440136551355\\
1.1	29.8994439277092\\
1.2	29.9031887629185\\
1.3	29.9905792250483\\
1.4	29.8424324944621\\
1.5	29.8536442469341\\
1.6	29.8269212135024\\
1.7	29.8860254628935\\
1.8	29.8550227496348\\
1.9	29.9187223869966\\
2	29.8732687651085\\
2.1	29.8429214658471\\
2.2	29.8852924894355\\
2.3	29.9234429176317\\
2.4	29.7987714454821\\
2.5	29.8427059226136\\
2.6	29.9268353907351\\
2.7	29.9082619796268\\
2.8	29.9157084143231\\
2.9	29.917644225447\\
3	29.8531848649264\\
3.1	29.8216189532455\\
3.2	29.954456293972\\
3.3	29.8539947979522\\
3.4	29.9717836447014\\
3.5	29.9793352215058\\
3.6	29.9574445232326\\
3.7	29.9221668949465\\
3.8	29.8529642202597\\
3.9	29.9253814587457\\
4	29.8585158294713\\
4.1	29.9312268157462\\
4.2	29.8867818345868\\
4.3	29.8548832964391\\
4.4	29.8327445121527\\
4.5	29.864094089698\\
4.6	29.9294031702885\\
4.7	29.8555522103354\\
4.8	29.920183596465\\
4.9	29.8910287356865\\
5	29.8993821031991\\
};
\addlegendentry{$K=10$, SSA};

\end{axis}
\end{tikzpicture}%
 \label{fig:thirdOrder:secondMoment} } 
    \caption[]{Reversible dimerization system~\eqref{eq:reversible:dimerization} with reaction constants $K = \nicefrac{k_2}{k_1}$: Comparison of the zero-information moment closure method \eqref{eq:finite:ODE}, solved using Algorithm~\hyperlink{algo:1}{1} and the average of $10^6$ SSA trajectories. The initial conditions are $\mathcal{M}_0 = 10$ and $\mathcal{D}_0=0$ and the regularization term $\kappa=0.01$.}
    \label{fig:plotExperimentMeanVar}
\end{figure}

To illustrate this point, we consider a reversible dimerisation reaction where two monomers ($\mathcal{M}$) combine in a second-order monomolecular reaction to form a dimer ($\mathcal{D}$); the reverse reaction is first order and involves the decomposition of the dimer into the two monomers. This gives rise to the chemical reaction system
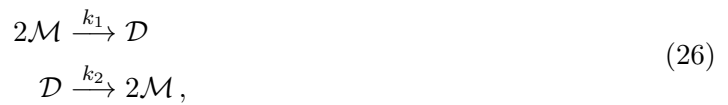
\begin{equation} \label{eq:reversible:dimerization}
\begin{aligned}
2\mathcal{M} & \overset{k_1}{\longrightarrow} \mathcal{D} \\
\mathcal{D}&\overset{k_2}{\longrightarrow} 2\mathcal{M} \, ,
\end{aligned}
\end{equation}
with reaction rate constants $k_1,k_2>0$.
Note that the system as described has a single degree of freedom since $\mathcal{M} = 2\mathcal{D}_0 - 2\mathcal{D} + \mathcal{M}_0$, Where $\mathcal{M}$ denotes the count of the monomers, $\mathcal{D}$ the count of dimers, and $\mathcal{M}_0$ and $\mathcal{D}_0$ the corresponding initial conditions.
Therefore, the matrices can be reduced to include only the moments of one component as a simplification and as such the zero-information closure function \eqref{eq:closure:function} consists of solving a one-dimensional entropy maximization problem such as given by \eqref{eq:main:problem:special case}, where the support are the natural numbers (upper bounded by $\mathcal{M}_0+\mathcal{D}_0$ and hence compact). 
For illustrative purposes, let us look at a second order closure scheme, where the corresponding moment vectors are defined as $\mu=(1, \langle \mathcal{M}\rangle, \langle \mathcal{M}^2\rangle)\transp\in\R^3$ and $\zeta = \langle \mathcal{M}^3 \rangle \in\R$ and the corresponding matrices are given by
\begin{align*}
A = \begin{pmatrix*}[c]
0 & 0 & 0 \\
k_2 S_0 & 2k_1-k_2 & -2k_1 \\
2k_2 S_0 & \quad 2k_2(S_0 -1) - 4k_1 \quad\quad  & 8k_1 -2k_2
\end{pmatrix*}, \quad
B =  \begin{pmatrix*}[c]
0  \\
0  \\
-4k_1
\end{pmatrix*},
\end{align*}
where $S_0 = \mathcal{M}_0 + 2\mathcal{D}_0$. The simulation results, Figure~\ref{fig:plotExperimentMeanVar}, show the time trajectory for the average and the second moment of the number of $\mathcal{M}$ molecules in the reversible dimerization model~\eqref{eq:reversible:dimerization}, as calculated for the zero information closure~\eqref{eq:finite:ODE} using Algorithm~\hyperlink{algo:1}{1}, for a second-order closure as well as a third-order closure. To solve the ODE~\eqref{eq:finite:ODE} we use an explicit Runge-Kutta (4,5) formula (ode45) built into MATLAB.
The results are compared to the average of $10^6$ SSA \cite{ref:Wilkinson-06} trajectories. It can be seen how increasing the order of the closure method improves the approximation accuracy.


\section{Approximate dynamic programming for constrained Markov decision processes} \label{sec:ADP}
In this section, we show that problem~\eqref{eq:main:problem} naturally appears in the context of approximate dynamic programming, which is at the heart of reinforcement learning (see \cite{ref:BerTsi-96, ref:BerDim-95, 2018arXiv180609460R} and references therein). We consider constrained Markov decision processes (MDPs) that form an important class of stochastic control problems with applications in many areas; see \cite{ref:Piunovskiy-97, ref:Altman-99} and the comprehensive bibliography therein. A look at these references shows that most of the literature concerns constrained MDPs where the state and action spaces are either finite or countable. Inspired by the recent work \cite{ref:MohSut-17}, we show here that the entropy maximization problem \eqref{eq:main:problem} is a key element to approximate constrained MDPs on general, possibly uncountable, state and action spaces.

\subsection{Constrained MDP problem formulation}
Consider a discrete-time constrained MDP $\big( S,A,\{A(s) : s\in S\},Q,c,d,\kappa \big),$
	where $S$ (resp.\ $A$) is a metric space called \emph{state space} (resp.\ \emph{action space}) and for each $s \in S$ the measurable set $A(s) \subseteq A$ denotes the set of \textit{feasible actions} when the system is in state $s\in S$. The \emph{transition law} is a stochastic kernel $Q$ on $S$ given the feasible state-action pairs in $\K:=\{(s,a):s\in S, a\in A(s)\}$. A stochastic kernel acts on real-valued measurable functions $u$ from the left as $Qu(s,a):= \int_{S}u(s') Q(\drv s'|s,a)$, for all $(s,a)\in \K,$ and on probability measures $\mu$ on $\K$ from the right as $\mu Q(B) := \int_{K}Q(B|s,a)\mu\big(\drv(s,a)\big)$, for all $B\in \mathcal{B}(S).$ The (measurable) function $c:\K \to\R_+$ denotes the so-called \emph{one-stage cost function}, and the (measurable) function $d:\K \to\R^q$ 
the one-stage constraint cost along with the preassigned budget level $\kappa\in\R^q$. 

In short, the MDP model described above may read as follows: When the system at the state~$s \in S$ deploys the action~$a \in A(s)$, it incurs the one-stage cost and constraint $c(s,a)$ and $d(s,a)$, respectively, and subsequently lands in the next state whose distribution is supported on $S$ and described via $Q(\cdot|s,a)$. 
We consider the expected long-run average cost criteria\footnote{We refer the interested reader to \cite{ref:MohSut-17} for extension to the discounted cost.}, i.e., for any measurable function $\psi:\K \to\R^p$ with $p\in\{1,q\}$, we define
\begin{equation*}
J_{\psi}(\pi,\nu)  := \limsup_{n\to\infty} \frac{1}{n} \mathds{E}^{\pi}_{\nu}\left( \sum_{t=0}^{n-1}\psi(s_{t},a_{t}) \right). 
\end{equation*} 
\noindent Using the defintion above, the central object of this section is the optimization problem
\begin{equation}\label{eq:CMDP}
J^\star: = \left\{ \begin{array}{rl}
\inf\limits_{\nu,\pi}	&J_{c}(\pi,\nu) \\
\text{s.t.}		&J_{d}(\pi,\nu) \leq \kappa \\
& \nu\in\mathcal{P}(S),\ \pi\in\Pi,
\end{array}\right.
\end{equation}
where $\Pi$ denotes the set of all control policies. We refer the reader to \cite{ref:Hernandez-96, ref:Hernandez-03, ref:Altman-99} for a detailed mathematical treatment of this setting. It is well known that the MDP problem~\eqref{eq:CMDP} can be stated equivalently as an infinite-dimensional linear program and its corresponding dual
\begin{subequations}
\begin{align}\label{eq:MDP:primal}
J_P^\star&: = \left\{ \begin{array}{rl}
\inf\limits_{\mu}	&\inprod{\mu}{c} \\
\text{s.t.}		&\mu(B\times A) = \mu Q(B) \quad \forall B\in\mathbb{B}(S)\\
& \inprod{\mu}{d}\leq \kappa \\
& \mu\in\mathcal{P}(\K),
\end{array}\right.\\
\label{eq:MDP:dual}
J_D^\star &: = \left\{ \begin{array}{rl}
\sup\limits_{u,\rho,\gamma}	&\rho - \gamma\transp \kappa \\
\text{s.t.}		&\rho + u(x)-Qu(x,a) \leq c(x,a) + \gamma\transp d(x,a)\quad \forall (x,a)\in\K\\
& \inprod{\mu}{d}\leq \kappa \\
& u\in, \rho\in\R, \gamma\in\R^q_+.
\end{array}\right.
\end{align}
\end{subequations}
The following regularity assumption is required in order to ensure that the solutions are well posed and that equivalence between \eqref{eq:CMDP} and the LPs \eqref{eq:MDP:primal}, \eqref{eq:MDP:dual} holds.
	\begin{myass}[Control model] \label{a:CM} We stipulate that
		\begin{enumerate}[(i)]
			\item \label{a:CM:K} the set of feasible state-action pairs is the unit hypercube $\K = [0,1]^{\dim(S\times A)}$;
			\item \label{a:CM:Q} the transition law $Q$ is Lipschitz continuous, i.e., there exists $L_Q>0$ such that for all $k, k'\in \K$ and all continuous functions $u$  
			$$|Qu(k) - Qu(k')| \leq L_Q \|u\|_\infty \|k - k'\|_{\ell_\infty};$$
			\item \label{a:CM:cost} the cost function $c$ is non-negative and Lipschitz continuous on $\K$ and $d$ is continuous on $\K$.
		\end{enumerate}
	\end{myass}
Under this assumption, strong duality between the linear programs \eqref{eq:MDP:primal} and \eqref{eq:MDP:dual} holds (i.e., the supremum and infimum are attained and $J_P^\star=J_D^\star$). Moreover, the LP formulation is equivalent to the original problem~\eqref{eq:CMDP} in the sense that $J^\star = J_P^\star=J_D^\star$, see \cite[Theorem~5.2]{ref:Hernandez-03}.
	
Finding exact solutions to either \eqref{eq:MDP:primal} or \eqref{eq:MDP:dual} generally is impossible as the linear programs are infinite dimensional. This challenge has given rise to a wealth of approximation schemes  in the literature under the names of approximate dynamic programming. Typically, one restricts decision space in \eqref{eq:MDP:dual} to a finite dimensional subspace spanned by basis functions $\{ u_i\}_{i=1}^n\subset \mathcal{L}(S)$ denoted by $\mathbb{U}_n:=\{\sum_{i=1}^n \alpha_i u_i \ : \ \|\alpha\|_2 \leq \theta\}$. Motivated by \cite{defarias:vanroy:04, ref:MohSut-17} we then approximate the solution $J^\star$ by 
\begin{equation} \label{eq:JPn}
J_{P,n}^\star: = \left\{ \begin{array}{rl}
\inf\limits_{\mu}	& \inprod{\mu}{c} +\theta \| \mathcal{T}_n \mu - e\|_2 \\
\text{s.t.}		& \inprod{\mu}{d}\leq \kappa \\
& \mu\in\mathcal{P}(\K),
\end{array}\right.
\end{equation}
where the operator $\mathcal{T}_n:\mathcal{P}(\K)\to\R^{n+1}$ is defined as $(\mathcal{T}_n \mu)_1 = -1, (\mathcal{T}_n \mu)_{i+1}=\inprod{Qu_i-u_i}{\mu}$,  $i=1,\hdots,n$ and $e:=(-1,0,\hdots,0)\in\R^{n+1}$. The optimization problem~\eqref{eq:JPn} can be solved with an accelerated first-order method provided by Algorithm~\hyperlink{algo:2}{2}, stated below, where each iteration step involves solving a problem an entropy maximization problem of the form~\eqref{eq:main:problem}. We define
\begin{align}
&c_{\alpha,\zeta}:=\frac{1}{\zeta}(c  -\alpha_0 + \sum_{i=1}^{n}\alpha_{i}(Qu_i-u_i)), \nonumber  &&\mathcal{Y}:=\{\mu\in \mathcal{P}(\K) \ : \ \inprod{\mu}{d}\leq \kappa\},  \nonumber \\
&\T\big(q,\alpha):= (\alpha - q)\min\{1,\theta \|q-\alpha\|_2^{-1}\}, &&\yeta(\alpha):=\arg \min\limits_{y\in\mathcal{Y}} \left\{ \KL{y}{\lambda} + \inprod{y}{c_{\alpha,\zeta}} \right\}.  \label{eq:entropymax}
\end{align}
 \begin{table}[!htb]
\centering 
\begin{tabular}{c}
  \Xhline{3\arrayrulewidth}  \hspace{1mm} \vspace{-3mm}\\ 
\hspace{0mm}{\bf{\hypertarget{algo:2}{Algorithm 2: } }} Approximate dynamic programming scheme \hspace{40.5mm} \\ \vspace{-3mm} \\ \hline \vspace{-0.5mm}
\end{tabular} \\
\vspace{-5mm}
 \begin{flushleft}
  {\hspace{3mm}{\bf{Input:}} $n, k \in\N$, $\zeta, \theta >0$, and $w^{(0)} \in \R^{n+1}$ such that $\|w^{(0)}\|_2\leq \theta$}
 \end{flushleft}
 \vspace{-6mm}
 \begin{flushleft}
  {\hspace{3mm}\bf{For $0\leq \ell\leq k$ do}}
 \end{flushleft}
 \vspace{-8mm}
 
  \begin{tabular}{l l}
\hspace{35mm}{\bf Step 1: } & \!\!\!\!Define $r^{(\ell)} :=\frac{\zeta}{4n} (e-\mathcal{T}_n \yeta(w^{(\ell)}))$ \\
\hspace{35mm}{\bf Step 2: } & \!\!\!\!Let  $z^{(\ell)} :=\T\big(\sum_{j=0}^{\ell} \frac{j+1}{2} r^{(j)}, 0\big)$ and $\beta^{(\ell)} =\T\big(r^{(\ell)}, w^{(\ell)}\big)$\\
\hspace{35mm}{\bf Step 3: } & \!\!\!\!Set $w^{(\ell+1)}=\frac{2}{\ell+3} z^{(\ell)} + \frac{\ell+1}{\ell+3} \beta^{(\ell)}$\\
  \end{tabular}
   \begin{flushleft}
  {\hspace{3mm}{\bf{Output:}} \quad $J^{(k)}_{n,\zeta} := \inprod{c}{\hat{y}_\zeta} + \theta \|\mathcal{T}_n \hat{y}_\zeta - e \|_{2} \quad \text{with}\quad \hat{y}_\zeta :=\sum_{j=0}^{k}\frac{2(j+1)}{(k+1)(k+2)} \yeta(w^{(j)})$}
 \end{flushleft}
   \begin{flushleft}
  \vspace{-15mm}
 \end{flushleft}  
\begin{tabular}{c}
\hspace{1mm} \phantom{ {\bf{Algorithm:}} Optimal Scheme for Smooth $\&$ Strongly Convex Optimization}\hspace{15mm} \\ \vspace{-1.0mm} \\\Xhline{3\arrayrulewidth}
\end{tabular}
\end{table}

Steps 2 and 3 of Algorithm~\hyperlink{algo:2}{2} are simple arithmetic operations. Step 1 relies on the solution to the optimization problem \eqref{eq:entropymax} that can be reduced to \eqref{eq:main:problem}. To see this, note that the additional linear term in the objective function can be directly integrated into the analysis of Section~\ref{sec:rate:distortion}. We provide the explicit construction in Section~\ref{section:inventory:management:system} for an inventory management system. 
We are interested in establishing bounds on the quality of the approximate solution $J_{n,\zeta}^{(k)}$ obtained by Algorithm~\hyperlink{algo:2}{2}. The results of \cite[Theorem~3.3 and 5.3]{ref:MohSut-17} allow us to obtain the following bound.

\begin{theorem}[Approximation error] \label{thm:approx:error:MDP}
Under Assumption~\ref{a:CM}, Algorithm~\hyperlink{algo:2}{2} provides an approximation to \eqref{eq:CMDP} with the following error bound
\begin{equation*}
|J_{n,\zeta}^{(k)}-J^\star| \leq (1 + \max\{L_Q, 1\}) \|u\opt - \proj_{\mathbb{U}_n}(u\opt)\|_L + \bigg( \frac{4n^2\theta}{k^2\zeta} + \dim(\K) \zeta \max\{\log\left( \beta /\zeta \right), 1\} \bigg),
\end{equation*}
where $\beta:=\frac{e}{\dim(\K)}(\theta \sqrt{n} (\max\{L_Q,1\}+1)+\|c\|_L)$.
\end{theorem}

Note that the bound depends on the projection residual $\|u\opt - \proj_{\mathbb{U}_n}(u\opt)\|_L$, where the projection mapping is defined as $\proj_{\mathbb{U}_n}(u\opt):=\arg\min_{u\in \mathbb{U}_n}\| u^\star - u\|_2$, the parameters of the problem (notably the dimensions of the state and action spaces, the stage cost, and the Lipschitz constant of the kernel $L_Q$), and the design choices of the algorithm (the number of basis functions $n$, norm bound $\theta$ and $\zeta$).

\begin{remark}[Tuning parameters] \
\begin{enumerate}[(i)]
\item The residual error $\|u\opt - \proj_{\mathbb{U}_n}(u\opt)\|_L$ can be approximated by leveraging results from the literature on universal function approximation. Prior information about the value function $u^\star$ may offer explicit quantitative bounds. For instance, for MDP under Assumption~\ref{a:CM} we know that $u^\star$ is Lipschitz continuous. For appropriate choice of basis functions, we can therefore ensure a convergence rate of ${n}^{-1/\dim(S)}$, see for instance \cite{ref:Farouki-12} for polynomials and \cite{ref:Olver-09} for the Fourier basis functions. 
\item The regularization parameter $\theta$ has to be chosen such that $\theta>\|c\|_L$. An optimal choice for $\theta$ and $\zeta$ is described in \cite[Remark~4.6 and Theorem~5.3]{ref:MohSut-17}.  
\end{enumerate}
\end{remark}

\subsection{Inventory management system} \label{section:inventory:management:system}

Consider an inventory model in which the state variable describes the stock level $s_t$ at the beginning of period $t$. The control or action variable $a_t$ at $t$ is the quantity ordered and immediately supplied at the beginning of period $t$, and the ``disturbance" or ``exogenous" variable $\xi_t$ is the demand during that period. We assume $\xi_t$ to be i.i.d. random variables following an exponential distribution with parameter $\lambda$.
The system equation \cite{ref:Hernandez-96} is
\begin{equation} \label{eq:inventory}
s_{t+1} = \max\{0,s_t + a_t - \xi_t\} =: (s_t + a_t - \xi_t)_+, \quad t=0,1,2, \hdots.
\end{equation}
We assume that the system has a finite capacity $C$. Therefore $S = A = [0, C]$ and since the current stock plus the amount ordered cannot exceed the system's capacity, the set of feasible actions is $A(s) = [0, C -s]$ for every $s\in S$.
Suppose we wish to maximize an expected profit for operating the system, we might take the net profit at stage $t$ to be
\begin{equation} \label{eq:cost:definition}
r(s_t, a_t, \xi_t):=v \min\{s_t + a_t,\xi_t\} - p a_t - h(s_t + a_t),
\end{equation}
which is of the form ``profit = sales - production cost - holding cost". In \eqref{eq:cost:definition}, $v$, $p$ and $h$ are positive numbers denoting unit sale price, unit production cost, and unit holding cost, respectively. To write the cost \eqref{eq:cost:definition} in the form of our control model \eqref{eq:CMDP}, we define
\begin{equation}
\begin{aligned}
c(s,a):=& \ \E{-r(s_t, a_t, \xi_t)| s_t=s, a_t=a}\\
=& \ -v (s+a) e^{ -\lambda(s+a) }-\frac{v}{\lambda}\left( 1- e^{ -\lambda(s+a) }\left( \lambda(s+a)+1 \right)\right) + p a + h(s + a).
\end{aligned}
\end{equation}
Note that non-negativity of the cost can be ensured by subtracting the term $2vC$ from \eqref{eq:cost:definition}.~We assume that there are regulatory constraints on the required stock level $s_t$, for example to avoid the risk of running into a shortage of a certain critical product. For simplicity, let us assume that the regulator enforces constraints on the long-term first and second moments of the stock $s_t$ in the following sense
\begin{align} \label{eq:bounds:inventory:regulator}
\limsup_{n\to\infty} \frac{1}{n} \mathds{E}^{\pi}_{\nu}\left( \sum_{t=0}^{n-1}s_t\right)\geq \ell_1, \quad \text{and} \quad  \limsup_{n\to\infty} \frac{1}{n} \mathds{E}^{\pi}_{\nu}\left( \sum_{t=0}^{n-1}s^2_t\right)\leq \ell_2, 
\end{align}
for given $\ell_1,\ell_2\in \R_+$, where we assume that $\ell_1^2<\ell_2$\footnote{Note that the enforced constraints \eqref{eq:bounds:inventory:regulator} imply the upper and lower bounds $\ell_1\leq \limsup_{n\to\infty} \frac{1}{n} \mathds{E}^{\pi}_{\nu}\left( \sum_{t=0}^{n-1}s_t\right)\leq \sqrt{\ell_2}$ and $\ell_1^2\leq \limsup_{n\to\infty} \frac{1}{n} \mathds{E}^{\pi}_{\nu}\left( \sum_{t=0}^{n-1}s^2_t\right)\leq \ell_2$.}. To express it in the form of our control model \eqref{eq:CMDP}, we define $d_1(s,a):= -s$, $d_2(s,a):=s^2$, and $\kappa:=(-\ell_1, \ell_2)\in\R^2$. From the described assumptions on the constants $\ell_1$ and $\ell_2$, it can be directly seen that the Slater point assumption described in Lemma~\ref{lem:zero:duality:gap} holds (consider the set $T:=\{ x\in\R^2 \ : \ x_1\geq \ell_1, x_2\leq \ell_2, x_1^2 \leq x_2 \}$).

In the following, we will describe how Algorithm~\hyperlink{algo:2}{2} and the respective error bounds given by Theorem~\ref{thm:approx:error:MDP} can be applied to the inventory management system. In particular we show how the computationally demanding part of Algorithm~\hyperlink{algo:2}{2}, given by \eqref{eq:entropymax}, is directly addressed by the methodology presented in Section~\ref{sec:rate:distortion}. To fulfill Assumption~\ref{a:CM}, we equivalently reformulate the above problem using the dynamics
\begin{equation}\label{eq:dynamics2}
s_{t+1} =  (\min\{C, s_t + a_t\} - \xi_t)_+, \quad t=0,1,2, \hdots,
\end{equation}
where  the admissible actions set is now the state-independent set $A=[0,C]$. Finally by normalizing the state and action variables through the definitions $\tilde{s}_t:=\frac{s_t}{C}$ and $\tilde{a}_t:=\frac{a_t}{C}$, we have
\begin{equation}\label{eq:dynamics2}
\tilde{s}_{t+1} =  \left(\min\{1, \tilde{s}_t + \tilde{a}_t\} - \frac{\xi_t}{C}\right)_+, \quad t=0,1,2, \hdots.
\end{equation}
Furthermore, it can be seen directly using Leibnitz' rule that the transition law is Lipschitz continuous and $L_Q\leq \sqrt{2}C\lambda$. It remains to argue how \eqref{eq:entropymax} can be addressed by the methodology presented in Section~\ref{sec:rate:distortion}. We introduce the linear operator $\mathcal{A}:\mathcal{P}(\K)\to \R^2$, defined by
$(\mathcal{A}\mu)_i:=\inprod{\mu}{d_i}$ for $i=1,2$. Then, the optimization problem \eqref{eq:entropymax} can be expressed as
\begin{align}
\text{(primal program)}: \quad J^{\star} &= \min\limits_{\mu\in\mathcal{P}(\K)} \Big \{  \KL{\mu}{\nu} -\inprod{\mu}{c_{\alpha,\zeta}} + \sup_{z\in\R^{M}}\left\{ \inprod{\mathcal{A}\mu}{z} - \sigma_{T}(z)\right\} \Big \}, \label{eq:primal:problem:inventory}\\
\text{(dual program)}: \quad J_{\mathsf{D}}^{\star} &= \sup_{z\in\R^{M}} \Big \{ - \sigma_{T}(z)  +   \min\limits_{\mu\in\mathcal{P}(\K)} \left\{ \KL{\mu}{\nu}  + \inprod{\mu}{\mathcal{A}^*z-c_{\alpha,\zeta}}\right\} \Big \} \, ,  \label{eq:dual:problem:inventory}
\end{align}
which apart from the additional linear term are in the form of \eqref{eq:primal:problem} and \eqref{eq:dual:problem}. Due to our assumptions on $\ell_1$ and $\ell_2$ described above, there exists a strictly feasible solution to \eqref{eq:entropymax}, i.e., $\mu_0\in\mathcal{P}(\K)$ such that $\mathcal{A}\mu_0\in T$ and $\delta:=\min_{y\in  T^c} \| \mathcal{A}\mu_0 - y\|_2 >0$. Hence, the method from Section~\ref{sec:rate:distortion} is applicable.

\textbf{Numerical Simulation.}
For a given set of model parameters ($C=1$, $\lambda=\frac{1}{2}$, $v=1$, $p=\frac{1}{2}$, $h=\frac{1}{10}$) we consider four different scenarios:
\begin{enumerate}[(i)]	
\item Unconstrained case (i.e., $\ell_1=0$ and $\ell_2=1$); \label{scenario:1}
\item $\ell_1=0.5$ and $\ell_2=0.4$; \label{scenario:2}
\item $\ell_1=0.5$ and $\ell_2=0.3$, which is strictly more constrained than scenario \eqref{scenario:2}; \label{scenario:3}
\item $\ell_1=0.1$ and $\ell_2=0.1$. \label{scenario:4}
\end{enumerate}

For each scenario we run Algorithm~\hyperlink{algo:2}{2} and plot the value of the resulting approximation $J_{n,\zeta}^{(k)}$ as a function of the number of iterations in Figure~\ref{fig:ADP}. In each iteration of Algorithm~\hyperlink{algo:2}{2} the variable \eqref{eq:entropymax} is required, which is computed with Algorithm~\hyperlink{algo:1}{1}. The following simulation parameters were used:
\begin{itemize}
  \setlength{\itemindent}{15mm}
\item[Algorithm~\hyperlink{algo:1}{1}:] $\eta_1=\eta_2 = 10^{-3}$, $1500$ iterations; 
\item[Algorithm~\hyperlink{algo:2}{2}:] $\zeta = 10^{-1.5}$, $\theta = 3$, $n=10$, Fourier basis $u_{2i-1}(s) = \frac{C}{2 i \pi}\cos\left(\frac{2 i \pi s}{C}\right)$ and $u_{2i}(s) = \frac{C}{2 i \pi}\sin\left(\frac{2 i \pi s}{C}\right)$ for $i=1,\hdots, \frac{n}{2}$.
\end{itemize}

As shown in Figure~\ref{fig:ADP}, the value of $J_{n,\zeta}^{(k)}$ converges at around $1000$ iterations of Algorithm~\hyperlink{algo:2}{2}.\footnote{1000 iterations  of Algorithm~\hyperlink{algo:2}{2} took around 5.5 hours with Matlab on a laptop with a 2.2 GHz Intel Core i7 processor.} The fact that scenario \eqref{scenario:2} is a (strict) relaxation in terms of the constraints compared to scenario \eqref{scenario:3} is visualized by the numerical simulation as the expected profit of scenario \eqref{scenario:2} is continuously higher compared to scenario \eqref{scenario:3}. Figure~\ref{fig:ADP} also indicates that in Scenario~\eqref{scenario:4} the constraints are the most restrictive.

\begin{figure}[!htb] 
\centering
   \input{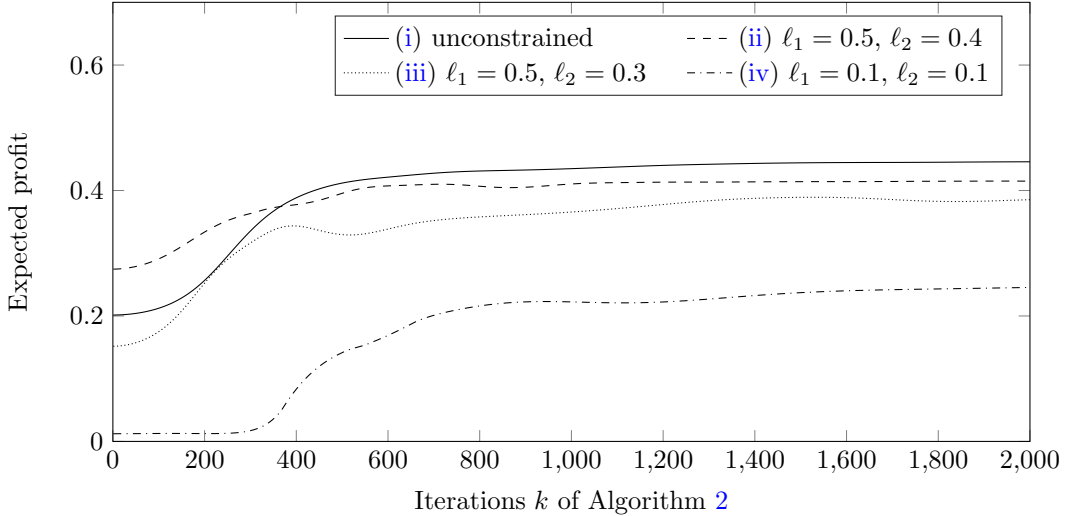} \label{fig:secondOrder} 
    \caption[]{The expected profit of the inventory system, approximated by $-J^{(k)}_{n,\zeta}$, resulting from Algorithm~\hyperlink{algo:2}{2} is displayed for four different scenarios \eqref{scenario:1}-\eqref{scenario:4} representing four different constraints on the inventory system.}
    \label{fig:ADP}
\end{figure}


 \section{Conclusion and future work} \label{sec:conclusion}
We presented an approximation scheme to a generalization of the classical problem of estimating a density via a maximum entropy criterion, given some moment constraints. The key idea used is to apply smoothing techniques to the non-smooth dual function of the entropy maximization problem, that enables us to solve the dual problem efficiently with fast gradient methods. Due to the favorable structure of the considered entropy maximization problem, we provide explicit error bounds on the approximation error as well as a-posteriori error estimates.

The proposed method requires one to evaluate the gradient \eqref{eq:gradient:approximation} in every iteration step, which, as highlighted in Section~\ref{sec:gradient:evalutaion}, in the infinite-dimensional setting involves an integral. As such the method used to compute those integrals has to be included to the complexity of the proposed algorithm and, in higher dimensions, may become is the dominant factor.~Therefore, it would be interesting to investigate this integration step in more detail.~Two approaches, one based on semidefinite programming and another invoking Quasi-Monte Carlo integration techniques, are briefly sketched. What remains open is to quantify the accuracy required in the gradient approximations, which could be done along the lines of \cite{ref:Devolder-13}.
Another potential direction, would be to test the proposed numerical method in the context of approximating the channel capacity of a large class of memoryless channels~\cite{TobiasSutter15}, as mentioned in the introduction.
 
 Finally it should be mentioned that the approximation scheme proposed in this article can be further generalized to quantum mechanical entropies. In this setup probability mass functions are replaced by density matrices (i.e., positive semidefinite matrices, whose trace is equal to one). The von Neumann entropy of such a density matrix $\rho$ is defined by $H(\rho):=-\mathrm{tr}(\rho \log \rho)$, which reduces to the (Shannon) entropy in case the density matrix $\rho$ is diagonal. Also the relative entropy can be generalized to the quantum setup~\cite{umegaki62} and general treatment of our approximation scheme, its analysis can be lifted to the this (strictly) more general framework. As demonstrated in~\cite{SSER16}, (quantum) entropy maximization problems can be used to efficiently approximate the classical capacity of quantum channels.

\begin{appendix}
\section{Detailed proof of Theorem~\ref{thm:main:result:inf:dim}} \label{app:detailed:proof}
Given the bounds $\| z^\star \|_2 \leq \frac{C}{\delta}$ and $ F(z) \leq J^\star \leq \KL{\mu_0}{\nu}=C$ for all $z \in \R^M$ as argued above, we show how the error bounds of Theorem~\ref{thm:main:result:inf:dim} follow from \cite{ref:devolder-12}. The dual $\varepsilon$-optimality \eqref{eq:thm:dual:optimality:cts} is derived from \cite[Equation~(7.8)]{ref:devolder-12}, that in our setting states
\begin{equation} \label{app:dual:optimality}
0\leq J^\star - F(\hat{z}_{k})\leq \frac{3\varepsilon}{4} + 5 \left(F(z^\star) - F(0)+\frac{\varepsilon}{2}\right)e^{-\frac{k}{2}\sqrt{\frac{\eta_2}{L(\eta)}}}.
\end{equation}
Using the parameters as defined in Theorem~\ref{thm:main:result:inf:dim}, i.e.,
\begin{align} \label{app:param}
 &C:=\KL{\mu_0}{\nu}, \qquad D :={1 \over 2} \max_{x \in T} \|x\|_2, \qquad \eta_{1}(\varepsilon) :=\frac{\varepsilon}{4D},  \qquad  \eta_{2}(\varepsilon) :=\frac{\varepsilon \delta^2}{2C^2},
 \end{align}
the fact that $F(z^\star) - F(0) = F(z^\star) \leq C$, and the Lipschitz constant $L(\eta) = \frac{1}{\eta_1}+\|\mathcal{A}\|^2 + \eta_2$ derived in Lemma~\ref{lem:lipschitz:cts:gradient}, inequality \eqref{app:dual:optimality} ensures $0\leq J^\star - F(\hat{z}_{k})\leq\varepsilon$
if 
\begin{equation*}
k\geq N_1(\varepsilon):=2 \left( \sqrt{\frac{8DC^2}{\varepsilon^2 \delta^2}+\frac{2\|\mathcal{A}\|^2 C^2}{\varepsilon \delta^2}+1}\right) \ln\left(\frac{10(\varepsilon +2C)}{\varepsilon}\right). \end{equation*}
The primal $\varepsilon$-optimality bound \eqref{eq:thm:primal:optimality:cts} following the derivations in \cite{ref:devolder-12} is implied if $k$ is chosen large enough such that $\|\nabla F_\eta(z_k)\|_2 \leq \frac{2\varepsilon \delta}{C}$. According to \cite[Equation~(7.11)]{ref:devolder-12} the gradient can be bounded by
\begin{equation}\label{app:primal:optimality}
\|\nabla F_\eta(z_k)\|_2 \leq \sqrt{4 L(\eta) \left( F(z^\star) - F(0) +\frac{\varepsilon}{2} \right)} e^{-\frac{k}{2}\sqrt{\frac{\eta_2}{L(\eta)}}} + 2\sqrt{3}\eta_2 \frac{C}{\delta}.
\end{equation}
Again using the parameters and constants as we did for the dual $\varepsilon$-optimality bound we find that $\|\nabla F_\eta(z_k)\|_2 \leq \frac{2\varepsilon \delta}{C}$ if
\begin{equation*}
k\!\geq \! N_2(\varepsilon)\!:=\!2\! \left(\! \sqrt{\frac{8DC^2}{\varepsilon^2 \delta^2}\!+\!\frac{2\|\mathcal{A}\|^2 C^2}{\varepsilon \delta^2}\!+\!1}\!\right)\! \ln\!\left(\! \frac{C}{\varepsilon \delta(2-\sqrt{3})}\sqrt{4\left( \frac{4D}{\varepsilon}\!+\!\|\mathcal{A}\|^2 \!+\! \frac{\varepsilon \delta^2}{2C^2} \right)\!\left( C \!+\!\frac{\varepsilon}{2} \right)} \right).
\end{equation*}
The primal $\varepsilon$-feasiblility \eqref{eq:thm:primal:feasibility:cts} finally directly follows from \cite[Equation~(7.14)]{ref:devolder-12} and the bound $\|z^\star\|_2 \leq \frac{C}{\delta}.$ 
\end{appendix}

\section*{Acknowledgments}
The authors thank Andreas Milias-Argeitis for helpful discussions regarding the moment closure example.
TS and JL acknowledge "SystemsX" under the grant "SignalX".
DS acknowledges support by the Swiss National Science Foundation (SNSF) via the National Centre of
Competence in Research "QSIT" and by the Air Force Office of Scientific Research (AFOSR) via
grant FA9550-16-1-0245.
PME was supported by the Swiss National Science Foundation under grant "P2EZP2 165264".

\vskip 0.2in
\bibliography{bibliofile}

\end{document}